\documentclass[11pt]{article}
\usepackage{graphicx,color,amsfonts,amsmath,amssymb,enumerate}
\usepackage{amsthm,natbib,comment,bm}
\usepackage[pdfborder={0 0 0}]{hyperref}
\hypersetup{
  urlcolor = black,
  pdfauthor = {Marcel Nutz, Ruodu Wang, Zhenyuan Zhang},
  pdfkeywords = {Martingale transport; backward Monge map; Strassen's theorem},
  pdftitle = {Martingale Transports and Monge Maps},
  pdfsubject = {Martingale Transports and Monge Maps},
  pdfpagemode = UseNone
}
\usepackage{dsfont}
\usepackage{mathrsfs}
\usepackage{tikz-qtree}
\usepackage{xr}
\usepackage{threeparttable}
\usepackage{amssymb}
\usepackage{graphicx}
\usepackage{amsmath}
\usepackage{bm}
\usepackage{tikz}
\usepackage{bbm}

\usepackage{caption}
\usepackage{pgfplots}
\usepackage{subcaption}
\usepgfplotslibrary{fillbetween}
\usetikzlibrary{patterns}
\usepackage{color}
\usepackage{dsfont} 
\usetikzlibrary{decorations.pathreplacing}

\usepackage{tkz-graph}
\usepackage{pgfplots}
\usetikzlibrary{calc,patterns}

\usepackage{pgfplots}
\usepackage{subcaption}
\usepgfplotslibrary{fillbetween}
\usetikzlibrary{patterns}
\usepackage{amssymb}
\usepackage{graphicx}
\usepackage{amsmath}
\usepackage{bm}
\usepackage{tikz}
\usepackage{color}
\usepackage{caption}
  
\usetikzlibrary{decorations.pathreplacing}

\usepackage{tkz-graph}
\usepackage{pgfplots}
\usetikzlibrary{calc,patterns}

\def\d{\mathrm{d}}
\def\laweq{\buildrel {\mathrm{law}} \over =}
\def\lawis{\buildrel {\mathrm{law}} \over \sim}

\newcommand{\var}{\mathrm{Var}}

\newcommand{\E}{\mathbb{E}}
\newcommand{\R}{\mathbb{R}}
\newcommand{\pib}{\M_M(\mu,\nu)}
\newcommand{\pim}{\M(\mu,\nu)}
\newcommand{\M}{\mathcal{M}}
\newcommand{\cP}{\mathcal{P}}
\newcommand{\F}{\mathcal{F}}

\newcommand{\N}{\mathbb{N}}
\renewcommand{\P}{\mathbb{P}}

\DeclareMathOperator\supp{supp}

\DeclareMathOperator\ri{ri}
\DeclareMathOperator\conv{conv}

\newcommand{\cx}{\le_{\rm cx}}
\newcommand{\mb}{\le_{\rm MM}}
\newcommand{\la}{\leftarrow}

\newcommand{\e}{\le_{\rm E}}

\newcommand{\bone}{ {\mathbbm{1}} }

\newcommand{\Id}{\mathrm{id}}
\renewcommand{\ge}{\geqslant}
\renewcommand{\le}{\leqslant}
\renewcommand{\geq}{\geqslant}
\renewcommand{\leq}{\leqslant}
\newcommand{\ep}{\xi}
\newcommand{\ee}{\varepsilon}
\DeclareMathOperator*{\bary}{bary}

\theoremstyle{plain}
\newtheorem{theorem}{Theorem}[section]
\newtheorem{corollary}[theorem]{Corollary}%
\newtheorem{lemma}[theorem]{Lemma}%
\newtheorem{proposition}[theorem]{Proposition}%

\theoremstyle{definition}
\newtheorem{definition}[theorem]{Definition}%
\newtheorem{example}[theorem]{Example}%
\newtheorem{conjecture}[theorem]{Conjecture}%

\newtheorem{remark}[theorem]{Remark}%

\theoremstyle{definition}

\numberwithin{equation}{section}
\renewcommand{\cite}{\citet}
\renewcommand{\cdots}{\dots}
\DeclareMathOperator*{\argmin}{arg\,min}

\DeclareMathOperator{\Unif}{Unif}
\usepackage{tikz}
\usetikzlibrary{arrows.meta,positioning}

\usepackage[onehalfspacing]{setspace}

\begin{document}

\title{Martingale Transports and Monge Maps}
\author{Marcel Nutz\thanks{Departments of Statistics and Mathematics, Columbia University, USA. Email: mnutz@columbia.edu. Research supported by NSF grants DMS-1812661, DMS-2106056.} \and 
Ruodu Wang\thanks{Department of Statistics and Actuarial Science, University of Waterloo, Canada. Email: wang@uwaterloo.ca. Research supported by NSERC grant RGPIN-2018-03823 and Canada Research Chairs CRC-2022-00141.  
} 
\and
Zhenyuan Zhang\thanks{Department of Mathematics, Stanford University, USA. Email: zzy@stanford.edu}
}
 \date{\today}

\maketitle

\begin{abstract}
It is well known that martingale transport plans between marginals $\mu\neq\nu$ are never given by Monge maps---with the understanding that the map is over the first marginal~$\mu$, or forward in time. Here, we change the perspective, with surprising results. We show that any  distributions $\mu,\nu$  in convex order with $\nu$ atomless admit a martingale coupling given by a Monge map over the \emph{second} marginal~$\nu$. Namely, we construct a particular coupling called the barcode transport. Much more generally, we prove that such ``backward Monge'' martingale transports are dense in the set of all martingale couplings, paralleling the classical denseness result for Monge transports in the Kantorovich formulation of optimal transport. Various properties and applications are presented, including a refined version of Strassen's theorem and a mimicking theorem where the marginals of a given martingale are reproduced by a ``backward deterministic'' martingale, a remarkable type of process whose current state encodes its whole history.
\end{abstract}

\vspace{0.9em}

{\small

}

\section{Introduction}

Martingale optimal transport was introduced by \cite{BHP13} in the discrete-time setting and \cite{GHT14} in continuous time. Since then, it has been an area of vigorous research thanks to its rich structures, connections with mathematical finance (see \cite{H11} and \cite{H17} for surveys) and the optimal Skorokhod embedding problem (see \cite{BCH17} and the literature thereafter), and analogies with classical transport theory (e.g., \cite{BJ16}, \cite{BNT17}).
Given probability measures $\mu,\nu$ on $\R$, a transport plan (or transport, or coupling) is the joint distribution of a random vector $(X,Y)$ with $X\lawis \mu$ and $Y\lawis \nu$. It is a martingale transport (MT) if in addition $\E[Y|X]=X$; that is, if $(X,Y)$ is a one-period martingale. We denote the set of transports by $\Pi(\mu,\nu)$ and its subset of martingale transports by $\pim$. %
Strassen's theorem states that $\pim$ is nonempty if and only if $\mu,\nu$ are in convex order, denoted $\mu \le_{\rm cx} \nu$. See Section~\ref{sec:2} below for detailed definitions.

In classical transport theory (without the martingale constraint), much attention has been devoted to transport plans given by Monge maps (transport maps); i.e., transports $(X,Y)$ where $Y=g(X)$ for some measurable function $g:\R\to\R$, or equivalently $\pi\in\Pi(\mu,\nu)$  of the form $\pi=(\Id_{\R},g)_{\#}\mu$ where $\#$ denotes pushforward. The existence of such Monge transports typically requires $\mu$ to be atomless (unless $\nu$ has atoms satisfying particular conditions). Under this natural requirement, it is known that the optimizers for numerous important optimal transport problems are indeed Monge, for instance, the quantile (or Fr\'echet--Hoeffding) coupling which minimizes the square-distance cost. Moreover, the set of all Monge transports is known to be weakly dense in $\Pi(\mu,\nu)$, which leads to the equivalence of the Kantorovich and Monge formulations of optimal transport: for any continuous and suitably integrable cost function~$c$, the value $\inf_{\pi\in\Pi(\mu,\nu)} \int c\,\d\pi$ remains the same if the infimum is only taken over the subset of Monge transports. See for instance \citet[Theorem 9.3]{A03} and \citet[Theorem B]{P07}, as well as the monographs \cite{V03,V09} and \cite{S15} for further background and numerous references. 

In the literature on martingale transport, Monge transports have been mentioned mostly\footnote{A notable exception, kindly pointed out to us by D.\ Kramkov, is the work of \cite{KX22} on a  Kyle-type equilibrium model of insider trading. There, a particular two-dimensional martingale $(X,Y)$ is shown to be of the form $(X_{1},X_{2})=(f_{1}(Y_{1},Y_{2}),f_{2}(Y_{1},Y_{2}))$ and that property is crucial for the interpretation of $(X_{1},X_{2})$ as the total order and price, respectively, of the equilibrium. In this problem, the law~$\nu$ of~$Y$ is prescribed whereas the law~$\mu$ of~$X$ is endogenous to the equilibrium. Remarkably, in our notation, $\M(\mu,\nu)$ is shown to be a singleton for that particular~$\mu$, which suggests that~$\mu$ has quite distinct properties (cf.\ Theorem~\ref{th:uniq}).} to state that they are uninteresting: because any deterministic martingale is constant, a martingale transport can only be of the form~$(X,g(X))$ if $g$ is the identity map. In that case, $\mu=\nu$, and $(X,X)$ is the only martingale coupling. In the martingale setting, one may think automatically along the forward-in-time direction $\mu\to\nu$ that is natural for adapted stochastic processes. In this paper, we change the perspective and look backward in time: nothing obvious precludes the existence of non-trivial Monge maps over the second marginal; that is, martingales $(X,Y)$ of the form $(f(Y),Y)$, or martingale laws $\pi=(f,\Id_{\R})_{\#}\nu$. The name ``backward Monge martingale transport'' seems descriptive but lengthy, and as the ``forward'' version is uninteresting, we simply say \emph{Monge martingale transport} (MMT). Their collection is denoted $\pib$.

This paper is dedicated to the theory of Monge martingale transports as well as their implications. Given marginals $\mu \le_{\rm cx} \nu$, it is not obvious if an MMT exists---apart from the trivial fact that atoms in~$\nu$ often preclude the existence of any Monge transport (martingale or not) from~$\nu$ to~$\mu$. Of all the martingale couplings that have been described in the literature, we are not aware of one that is Monge for reasonably generic marginals. Assuming that~$\nu$ is atomless, we prove in Theorem~\ref{thm1} that $\pib$ is never empty: we construct a particular MMT that we call the \emph{barcode transport}, a name derived from its pictorial representation (see Figure~\ref{fig:1} on page~\pageref{fig:1}).\footnote{Strictly speaking, the {barcode transport} is constructed using the \emph{left-curtain transport}, whereas using the \emph{right-curtain transport} would yield a different {barcode transport}. However, for notational convenience, we simply call it a {barcode transport} instead of a {left-barcode transport}. } The basic idea is to decompose the marginals~$\mu$ and~$\nu$ into countably many pieces (the \emph{bars} of the barcode) that can be coupled by MMTs more easily, and then aggregate. As an auxiliary result, we provide a novel structural description (Proposition~\ref{p1}) of the \emph{left-curtain transport}~$\pi_{\rm lc}$ prominently introduced by \cite{BJ16}; we show in particular that $\pi_{\rm lc}$ is Monge if the first marginal has more mass than the second marginal at any point of its support. While this condition is of course quite special, we can always construct a decomposition of the original marginals $\mu,\nu$ such as to satisfy the condition on each ``bar''.

The aforementioned construction is rather particular and one may wonder whether the barcode transport is just an isolated curious example. Our main result (Theorem~\ref{th:dense}) states that the set $\pib$ of Monge martingale transports is  weakly dense in the set $\pim$ of all martingale transports. This shows that there are many MMTs (for typical marginals) and, paralleling the aforementioned results in classical transport theory, that the value $\inf_{\pi\in\pim} \int c\,\d\pi$ of a martingale optimal transport problem remains the same if the infimum is only taken over the subset of Monge transports (Corollary~\ref{C1}), for any continuous and suitably integrable~$c$. We mention that a quite different (and maybe less direct) parallel was established in the Skorokhod embedding problem: \cite{BeiglbockNutzStebegg.21} show that the stopping times of the Brownian filtration that embed a given distribution are weakly dense in the set of randomized stopping times embedding the distribution.

While the above shows that standard optimal transport problems cannot distinguish $\pib$ from $\pim$, a natural characterization of $\pib$ within $\Pi(\mu,\nu)$ will be given in terms of  generalized (or ``weak'') transport costs in the sense of \cite{GRST17}. These are cost functions depending not only on the origin and destination points of a transport but directly on the kernel (conditional distribution) of the coupling. We show in Proposition~\ref{prop:opt} that $\pib$ is the set of minimizers for a class of such problems, in particular (with obvious abuse of notation)  
$$ \pib=\argmin_{(X,Y) \in \Pi(\mu,\nu)}\E \left[\E[Y|X]^2-\E[X|Y]^2\right]-2\E[XY].$$

We also discuss in detail the uniqueness of MMT (Theorem~\ref{th:uniq}) which is equivalent to the uniqueness of MT, and happens only in very particular circumstances that we characterize in terms of so-called shadows. If both marginals $\mu,\nu$ are atomless, the only case with uniqueness is $\mu=\nu$.

Several applications of MMTs are presented. The first is a refinement of Strassen's theorem on~$\R$ (Theorem~\ref{th:refinedStrassen}) saying that if random variables $X$ and $Y$ on an atomless probability space satisfy $X\le_{\rm cx} Y$, then there exists a random  variable $X'\laweq X$ on the same space such that $X'=\E[Y|X']$ is a martingale. Thus~$Y$ is preserved, whereas the usual Strassen's theorem only guarantees a martingale $(X',Y')$ with the same marginal distributions but no particular relation to the original random variables~$(X,Y)$. 

Going further in a similar direction, we develop a mimicking theorem (in the sense of \cite{Gyongy.86}) with a  class of martingales that we call \emph{backward deterministic.} These are processes $(X_n)_{n\in \N}$ where $(X_j)_{j=1}^n$ is $\sigma(X_n)$-measurable. We may see this as a strengthening of the Markov property where the current state~$X_{n}$ already encodes the whole history $(X_j)_{j=1}^n$. A non-recombining binary tree is a good illustration. Our mimicking theorem (Corollary~\ref{coro:1}) states that given a martingale $(Y_n)_{n\in \N}$ 
with atomless marginals, there exists a backward deterministic martingale $(X_n)_{n\in \N}$ such that $X_n\laweq Y_n$ for all~$n$. 

The remainder of this paper is organized as follows. Section~\ref{sec:2} collects the main results on Monge martingale transports, as well as the result on the left-curtain transport to be used in the existence proof. In Section~\ref{se:applications} we discuss the applications regarding Strassen's theorem, the mimicking theorem with backward deterministic martingales, and the characterization of $\pib$ via generalized optimal transport. Section~\ref{P} contains the proofs for the main results stated in Section~\ref{sec:2}. We conclude with some comments and open problems in Section~\ref{sec:concluding}.

\section{Main results}
\label{sec:2}

Let $\mathcal P(\R)$ denote the set of Borel probability measures on~$\R$ with finite first moment. We say that $\mu,\nu\in\cP(\R)$ are in \emph{convex order}, denoted $\mu\cx\nu$, if $\int\phi\,\d\mu\leq \int\phi\,\d\nu$ for any convex function $\phi:\R\to\R$. This implies that $\mu,\nu$ have the same mean. We use the same notation for unnormalized finite measures; in that case $\mu,\nu$ must also have the same total mass. Occasionally we  write $X\cx Y$ for random variables $X,Y$ to indicate that their laws are in convex order. Recall from the Introduction that $\Pi(\mu,\nu)$ denotes the set of couplings, $\pim$ the subset of martingale couplings, and $\pib$ the further subset of (backward) Monge martingale transports. {We say that a measure $\pi$ is supported on a set $A$ if $A^{c}$ is a $\pi$-nullset. The topological support (that is, the smallest such set $A$ that is closed) may be different.}

Our first result yields the existence of a Monge martingale transport when the second marginal~$\nu$ is atomless. More generally, when $\nu$ has atoms, it establishes a martingale transport that is (backward) Monge outside the atoms---the Monge property on the atoms is  typically not achievable even without the martingale constraint. %

 \begin{theorem}[Existence]\label{thm1} Let $\mu,\nu\in\mathcal P(\R)$ satisfy $\mu\cx\nu$. There exists $\pi\in\pim$ and a Borel function $h:\R\to\R$ such that $\pi(T_{\rm rg}\cup T_{\rm atom})=1$, where
 \begin{enumerate}[(i)]
    \item $T_{\rm rg}=\{(h(y),y):y\in\R\}$;
    \item $T_{\rm atom}=\{(x,y):\nu(\{y\})>0\}$.
\end{enumerate}
In particular, if $\nu$ is atomless, $\pi$ is a Monge martingale transport.
 \end{theorem} 

To prove Theorem~\ref{thm1}, we will  explicitly construct a coupling called the \emph{barcode transport}. As mentioned in the Introduction, the basic idea is to decompose the marginals into countably many mutually singular parts---the bars of the barcode; cf.\ Figure~\ref{fig:1} below---tailored such that  the left-curtain transport~$\pi_{\rm lc}$ for each part is Monge outside of the atoms of~$\nu$. We thus need criteria for~$\pi_{\rm lc}$ to be Monge, and that is the purpose of the next result.

To state the definition of~$\pi_{\rm lc}$ given by \cite{BJ16}, we write $\mu\e\nu$ for finite measures $\mu,\nu$ with finite first moment if $\int\phi\,\d\mu\leq\int\phi\,\d\nu$ for any nonnegative convex function $\phi:\R\to\R$. If $\mu$ and $\nu$ have the same total mass, this is equivalent to $\mu\cx\nu$, but a quite different example is that $\mu\leq\nu$ (set-wise) implies $\mu\e\nu$. Given $\mu\e\nu$, the \emph{shadow}~$S^\nu(\mu)$ of $\mu$ in $\nu$ is defined 
as 
$S^\nu(\mu)=\min\{\eta: \, \mu\cx \eta \leq\nu\}$, where the minimum is taken in the partial order~$\cx$.
Intuitively, the shadow is formed by mapping each $\mu$-particle into~$\nu$ while greedily dispersing its mass as little as possible. See \citet[Lemma 4.6]{BJ16} for the wellposedness of~$S^\nu(\mu)$.

Given $\mu\cx\nu$, the \emph{left-curtain transport} $\pi_{\rm lc}\in \pim$  is uniquely defined by the property that it transports $\mu|_{(-\infty,x]}$ to its shadow $S^\nu(\mu|_{(-\infty,x]})$ for every $x\in\R$. It
 can be considered as the martingale analogue of the quantile coupling with respect to the convex order.
The ``forward'' structure of $\pi_{\rm lc}$ has been analyzed in detail by \cite{BJ16} as well as \cite{HT16} and \cite{HN19}; see also Section~\ref{S}. The following result describes the structure from the backward perspective and may be of independent interest. It states that in general, $\pi_{\rm lc}$ is supported on three sets: the reverse graph (or antigraph) $S_{\rm rg}$ of a function~$h:\R\to\R$, the diagonal $S_{\rm diag}$, and the atomic part $S_{\rm atom}$. For the proof of Theorem~\ref{thm1}, we will only use the second assertion, namely that if $\d\mu/\d(\mu+\nu)\geq 1/2$ $\mu$-a.e., the reverse graph can also capture the mass on $S_{\rm diag}$.

 \begin{proposition}[Structure of $\pi_{\rm lc}$]\label{p1}
  Let $\mu\cx\nu$. There exists a Borel function $h:\R\to\R$ such that the left-curtain transport $\pi_{\rm lc}$ satisfies $\pi_{\rm lc}(S_{\rm rg}\cup S_{\rm diag}\cup S_{\rm atom})=1$, where
  \begin{enumerate}[(i)]
      \item $S_{\rm rg}=\{(h(y),y):y\in\R\}$;
      \item $S_{\rm diag}=\{(x,x):x\in\R\}$;
      \item $S_{\rm atom}=\{(x,y):\nu(\{y\})>0\}$.
  \end{enumerate}
  If $\d\mu/\d(\mu+\nu)\geq 1/2$ $\mu$-a.e., then $\pi_{\rm lc}(S_{\rm rg}\cup S_{\rm atom})=1$ for some Borel~$h$. In particular, if in addition~$\nu$ is atomless, then $\pi_{\rm lc} \in \pib$.
 \end{proposition}

The second assertion is not directly a consequence of the first part as the function~$h$ may need to be redefined. We refer to Section~\ref{S} for further comments on~$\pi_{\rm lc}$. 

Figure~\ref{fig:1} illustrates the barcode transport and the left-curtain transport for Gaussian marginals. We observe that the left-curtain transport is not Monge in this case, and this arises due to the mass on~$S_{\rm diag}$ represented in light-gray over a subset of $\{\d\mu/\d(\mu+\nu)< 1/2\}$.

\begin{figure}[hbtp]
\begin{center}
    \begin{subfigure}[b]{0.48\textwidth}
    \centering

 \begin{tikzpicture}
\centering
\begin{axis}[width=1.1\textwidth,
      height=0.75\textwidth,
  axis lines=none,%
  domain=-10:10,
  xmin=-3, xmax=3,
  ymin=-2, ymax=2,
  samples=80,
]
  \addplot[name path=f,blue,domain={-3:3},mark=none]
    {1.3*exp(-1.69*x^2)} node[pos=.82,above] {$\mu$};
  \addplot[name path=g,blue,domain={-3:3},mark=none]
    {0};

    \addplot[black!10]fill between[of=f and g, soft clip={domain=-0.64:0.64}];
    \addplot[black!80]fill between[of=f and g, soft clip={domain=0.64:0.77}];
    \addplot[black!80]fill between[of=f and g, soft clip={domain=-0.77:-0.64}];
    \addplot[black!30]fill between[of=f and g, soft clip={domain=-0.93:-0.77}];
    \addplot[black!30]fill between[of=f and g, soft clip={domain=0.77:0.93}];
    \addplot[black!97]fill between[of=f and g, soft clip={domain=-1.11:-0.93}];
    \addplot[black!97]fill between[of=f and g, soft clip={domain=0.93:1.11}];
    \addplot[black!50]fill between[of=f and g, soft clip={domain=-1.33:-1.11}];
    \addplot[black!50]fill between[of=f and g, soft clip={domain=1.11:1.33}];
    \draw [decorate,decoration={brace,amplitude=5pt},xshift=0pt](237,200) -- (363,200) node [black,midway,yshift=1.7cm]  { \hspace{4cm}$\d\mu/\d(\mu+\nu)>1/2$};

 \path[<-, draw,  thick, dashed](300,220) 
      to[out=90, in=270] (420,326)
     node[right]{};

  \addplot[name path=p,blue,domain={-3:3}]
    {exp(-(x)^2)-2} node[pos=.83,above] {$\nu$};  
  \addplot[name path=q,blue,domain={-3:3}]
    {-2};
     \addplot[black!10]fill between[of=p and q, soft clip={domain=-0.77:0.77}];
    \addplot[black!80]fill between[of=p and q, soft clip={domain=-0.93:-0.77}];
    \addplot[black!80]fill between[of=p and q, soft clip={domain=0.77:0.93}];
    \addplot[black!30]fill between[of=p and q, soft clip={domain=-1.11:-0.93}];
    \addplot[black!30]fill between[of=p and q, soft clip={domain=0.93:1.11}];
    \addplot[black!97]fill between[of=p and q, soft clip={domain=-1.33:-1.11}];
    \addplot[black!97]fill between[of=p and q, soft clip={domain=1.11:1.33}];
    
    \addplot[black!50]fill between[of=p and q, soft clip={domain=-1.6:-1.33}];
    \addplot[black!50]fill between[of=p and q, soft clip={domain=1.33:1.6}];
 
\end{axis}
\draw[->](3,2.2)--(2.78,0);
\draw[->](3,2.2)--(3.15,0);
\draw[->](3.5,2.2)--(2.69,0.01);
\draw[->](3.5,2.2)--(4.1,0);
\end{tikzpicture}
\caption{The barcode transport}

    \end{subfigure}
    \begin{subfigure}[b]{0.48\textwidth}\centering
    \begin{tikzpicture}
\centering
\begin{axis}[width=1.1\textwidth,
      height=0.75\textwidth,
  axis lines=none,%
  domain=-10:10,
  xmin=-3, xmax=3,
  ymin=0, ymax=4,
  samples=80,
]
  \addplot[name path=f,blue,domain={-3:3},mark=none]
    {1.2*exp(-1.44*x^2)+2} node[pos=.82,above] {$\mu$};
  \addplot[name path=g,blue,domain={-3:3},mark=none]
    {2};
    \addplot[black!10]fill between[of=f and g, soft clip={domain=-3:-0.64}];    
    \addplot[black!55]fill between[of=f and g, soft clip={domain=-0.64:3}]; 

  \addplot[name path=p,blue,domain={-3:3}]
    {0.999*exp(-(x)^2)}  node[pos=.83,above] {$\nu$};
    
    \addplot[name path=r,blue,domain={-3:-0.64},mark=dots]
    {1.2*exp(-1.44*(x)^2)};   
  \addplot[name path=q,blue,domain={-3:3}]
    {0};
    \addplot[black!10]fill between[of=q and r, soft clip={domain=-3:-0.64}];    
    \addplot[black!55]fill between[of=r and p, soft clip={domain=-3:-0.64}];  
    \addplot[black!55]fill between[of=q and p, soft clip={domain=-0.64:3}];    
   
\end{axis}
\draw[->](2.1,2.2)--(2.1,0);
\draw[->](2.5,2.2)--(2.5,0);
\end{tikzpicture}
\caption{The left-curtain transport}
    \end{subfigure}

\caption{Comparison of the barcode transport and the left-curtain transport for Gaussian marginals. (a) The barcode transport consists of a collection of left-curtain transports represented by different shades. The map~$h$ follows the reverse of the indicated arrows. (b) The left-curtain transport is the identity on the light-gray area and does not admit a (backward) Monge map there.
}\label{fig:1}
\end{center}
\end{figure}
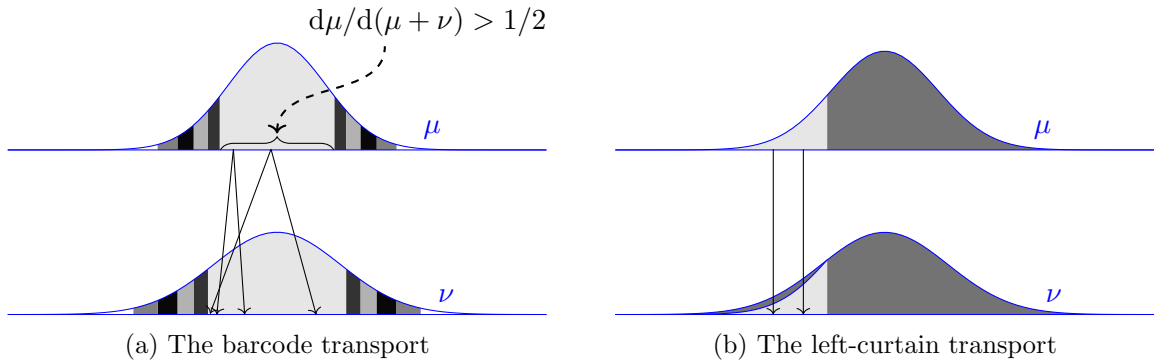

We continue with our main result, showing that  the set $\pib$ of Monge martingale transports is surprisingly rich.
 
 \begin{theorem}[MMTs are dense]\label{th:dense}
Let $\mu\cx\nu$ with $\nu$ atomless. Then $\pib$ is weakly dense in $\pim$. If $\mu$ is discrete, it is also dense for the $\infty$-Wasserstein topology.
 \end{theorem}

The proof is significantly more involved than the existence argument, hence we defer a sketch to Section~\ref{se:densityProof}. 
As a consequence of Theorem~\ref{th:dense}, we obtain the equivalence of the Kantorovich and (backward) Monge formulations for martingale optimal transport.

 \begin{corollary}\label{C1}
 Let $\mu\cx\nu$ with $\nu$ atomless. If $c:\R^{2}\to\R$ is continuous with $|c(x,y)|\leq a(x)+b(y)$ for some  $a\in L^1(\mu)$ and $b\in L^1(\nu)$, then
 $$\inf_{\pi\in\pib}\int_{\R\times\R} c(x,y)\,\pi(\d x,\d y)=\inf_{\pi\in\pim}\int_{\R\times\R} c(x,y)\,\pi(\d x,\d y).$$
 \end{corollary}
  
The final theorem of this section characterizes the uniqueness of MMT; that is, when $\pib$ is a singleton. We can already see from the denseness result in Theorem~\ref{th:dense} that this is equivalent to $\pim$ being a singleton (a more direct proof will be given in Section~\ref{P}). In terms of the marginals, uniqueness turns out to depend on the atoms of~$\mu$ and their shadows.

\begin{theorem}[Uniqueness]\label{th:uniq}
  Let $\mu\cx\nu$ with $\nu$ atomless. The following are equivalent:
    \begin{enumerate}[(i)]
     \item[(i)] The MT from $\mu$ to $\nu$ is unique. 
    \item[(ii)] The MMT from $\mu$ to $\nu$ is unique.
    \item[(iii)] Let $\mu_a:=\sum_{j\in \N} a_j\delta_{x_j}$ be the atomic part of $\mu$, where $\{x_j\}_{j\in\N}$ are distinct. Then the shadows $S^\nu(a_j\delta_{x_j})$, $j\in \N$ are mutually singular and $\mu-\mu_a=\nu-\sum_{j\in \N} S^\nu(a_j\delta_{x_j})$.
 \end{enumerate} 
\end{theorem}

\begin{remark}\label{remark:referee}
    As kindly pointed out by an anonymous referee, a further equivalent statement for Theorem \ref{th:uniq} can be formulated using
    the concept of irreducible components. 
For probability measures $\mu,\nu$ on $\R$ satisfying $\mu\cx\nu$, we let $u_\mu:\R\to\R,\,x\mapsto \int_\R|y-x|\mu(\d y)$ be the potential function of $\mu$, and similarly define $u_\nu$. Let $(I_k)_{1\leq k\leq N}$ be the (open) components of $\{u_\mu<u_\nu\}$ where $N\in\N\cup\{\infty\}$, and let $I_0=\R\setminus\bigcup_{k\geq 1}I_k$. Define $\mu_k=\mu|_{I_k}$, so that $\mu= \sum_{k\geq0}\mu_{k}$; this is called the irreducible decomposition of $\mu$ (which depends on $\nu$). By Theorem A.4 of \cite{BJ16}, there exists a unique decomposition $\nu=\sum_{k\geq0}\nu_{k}$ such that $\mu_0=\nu_0$ and $\mu_k\cx\nu_k$ for all $k$, and any $\pi\in\pim$ transports~$\mu_{k}$ to~$\nu_{k}$ for $k\in\N$ and $\mu_{0}$ to $\nu_{0}$  via the identity transport.
    Then we have the following equivalent condition for uniqueness of the MT:
    \begin{enumerate}
        \item[(iv)] Each $\mu_k,~k\in\N$ in the irreducible decomposition of $\mu$ is concentrated on a singleton. 
    \end{enumerate}
Indeed, (iv) implies the MT on each irreducible component is unique, and hence (i); the structure (iii) implies that $\mu_k=a_k\delta_{x_k},~k\in\N$ and $\mu_0=\nu-\mu_a$ define the irreducible decomposition of $\mu$, implying (iv).
The more general irreducible decomposition for probability measures on $\R^d$ instead of $\R$ will be discussed in Section \ref{sec:concluding}.
\end{remark}
   
 As a special case of Theorem~\ref{th:uniq}, if $\mu$ and $\nu$ are both atomless, uniqueness is equivalent to $\mu=\nu$. A nontrivial example with uniqueness is illustrated in Figure~\ref{fig:4}.
 
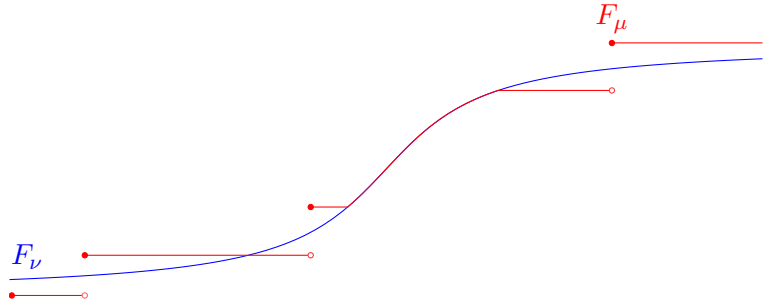
\begin{figure}[hbtp]\begin{center} 
\begin{tikzpicture}[scale=1]
\centering
\begin{axis}[width=.7\textwidth,
      height=0.35\textwidth,
  axis lines=none,%
  domain=-10:10,
  xmin=-5, xmax=5,
  ymin=-0.04, ymax=3.9,
  samples=80,
]
 
  \addplot[name path=p,blue,domain={-5:5}]
    {rad(atan(x))+1.5708}node[pos=0.022,above]{$F_\nu$};
    \addplot[red,domain={-5:-4}]{0};
    \addplot[red,domain=-4:-1]{0.5};
    \addplot[red,domain=-1:-0.5]{1.1};
    \addplot[red,domain=3:5]{3.1416}node[pos=0.001,above]{$F_\mu$};
\addplot[red,domain=-0.5:1.5]{rad(atan(x))+1.5708};
    \addplot[red,domain=1.5:3]{2.55};
    
    \filldraw[red] (3,4) circle (1pt);
    \filldraw[color=red!60, fill=red!0] (100,4) circle (1pt);
    
    \filldraw[red] (100,54) circle (1pt);
    \filldraw[color=red!80, fill=red!0] (400,54) circle (1pt);
    
    \filldraw[red] (400,114) circle (1pt);
    \filldraw[red] (800,318) circle (1pt);
    \filldraw[color=red!80, fill=red!0] (800,259) circle (1pt);
    
\end{axis}
\end{tikzpicture}\end{center}
\caption{Distribution functions of $\mu,\nu$  where the MMT (and MT) from $\mu$ to $\nu$ is unique} 
\label{fig:4}
\end{figure}

We conclude with simple examples illustrating subtleties that can arise when~$\nu$ is not atomless.

\begin{example}[MT exists; MMT does not] 
  Let $\mu$ and $\nu$ be two-point distributions satisfying $\mu\cx \nu$. Then there is a unique MT, as there is a unique distribution on two distinct points with a given mean. On the other hand, there is no  MMT unless $\mu=\nu$. In general, if $\mu,\nu$ are discrete and $\mathrm{card}(\cdot)$ denotes the cardinality of the support, the existence of an MMT implies $(2\,\mathrm{card} ((\mu -\nu)_+))\vee \mathrm{card}(\mu) \le \mathrm{card}(\nu)$.
\end{example}

\begin{example}[MMT is unique; MT is not] \label{ex:uniq}
Let $\mu$ be uniform on $\{2,5\}$ and $\nu$ be uniform on $\{0,3,4,7\}$.
   The unique MMT is given by  transporting $\{2\}$ to $\{0,4\}$
   and $\{5\}$ to $\{3,7\}$, 
   while it is easy to see that there exist many MTs. 
   \end{example}

\section{Applications and further properties}\label{se:applications}

\subsection{Refinement of Strassen's theorem}

The celebrated Strassen's theorem (\cite[Theorem 8]{S65}) shows that if two random variables $X$ and $Y$ satisfy $X\le_{\rm cx} Y$, then we can build $X'\laweq X$ and $Y'\laweq Y$ on another probability space such that $X'=\E[Y'|X']$. %
 Theorem~\ref{thm1} gives rise to the following refinement where $X'$ is built on the original space supporting~$Y$ and there is no need for an auxiliary random variable $Y'$. 

\begin{theorem}[Refinement of Strassen's theorem]\label{th:refinedStrassen} Let $X\cx Y$ be real-valued random variables on an atomless probability space $(\Omega,\F,\P)$. There exists a random variable $X'$ on $(\Omega,\F,\P)$ satisfying $X'\laweq X$ and $X'=\E[Y|X']$.
\end{theorem}

\begin{proof}
Let $\{y_n: n\in I\}\subseteq\R$ be the atoms of the distribution of $Y$, where $I$ is a countable set. As $(\Omega,\F,\P)$ is atomless, we can construct for each $n\in I$ a uniform random variable $U_{y_n}$ on $\{Y=y_n\}$ equipped with the restrictions of $\F$ and $\P$.  
It suffices to construct a random variable $X'$ that is $\sigma(Y,U_{y_n},n\in I)$-measurable such that $X'=\E[Y|X']$. By Theorem~\ref{thm1}, there exists a coupling $\pi$ of $X,Y$ supported on the union of a reverse graph $\{(h(y),y):y\in \R\}$ and $\bigcup_{n\in I}\{(x,y_n):x\in\R\}$. Let $F_{y_n}$ be the cdf of the conditional distribution of $\pi$ given $Y=y_n$ and let $F_{y_n}^\la$ denote its left-continuous inverse. We define
\begin{align*}
    X'(\omega):=\begin{cases}
    h(Y(\omega))&\text{ if }\omega\not\in \bigcup_{n\in I}\{Y=y_n\};\\
    F_{y_n}^\la(U_{y_n}(\omega))&\text{ if }\omega\in \{Y=y_n\}\text{ for some }n\in I.
    \end{cases}
\end{align*}
Then $X'$ is $\sigma(Y,U_{y_n},n\in I)$-measurable and the joint distribution of $(X',Y)$ is $\pi$.
\end{proof}

\begin{remark}\label{StrassenAndExistence}
Theorem~\ref{thm1} implies the existence of an MMT when the second marginal~$\nu$ is atomless. This statement can also be recovered from Theorem~\ref{th:refinedStrassen} by taking $\F=\sigma(Y)$, so that $X'$ must be a function of~$Y$. 
\end{remark}

A different way of framing those relations is to introduce a partial order on $\mathcal{P}(\R)$ via MMT. Noting that the convex order can be defined as $\mu\le_{\rm cx}\nu \Leftrightarrow \pim\neq\emptyset$, let us write  $\mu\mb\nu$ if $\pib\neq\emptyset$. This is indeed a partial order.

\begin{lemma}
 The binary relation $\mb$ is a partial order on $\mathcal P(\R)$. Moreover, $\mb$ implies $\le_{\rm cx}$.
 \end{lemma}
 \begin{proof}
 Clearly $\mb$ implies $\le_{\rm cx}$, hence reflexivity and antisymmetry of~$\mb$ follow from those of~$\le_{\rm cx}$.
 To show transitivity, let $\eta\mb\mu$ and $\mu\mb\nu$. By definition, there exist functions~$g$ and~$f$ such that given $Y\lawis \nu$ and $X\lawis \mu$, we have $\E[Y|f (Y)]=f(Y)\lawis\mu$ and $\E[X|g (X)]=g(X)\lawis\eta$. In particular, setting $X:=f(Y)$,
  $$
 \E[Y|g\circ f (Y)]=   \E\left[ \E[Y|f (Y)] |g\circ f(Y)\right] = \E[X|g(X)]=g(X)=g\circ f (Y),
 $$
 showing that $g\circ f$ is an MMT for $(\eta,\nu)$.
 \end{proof}

\begin{proposition}\label{pr:partialOrder}
  Let $\nu \in \mathcal P(\R)$ and $Y\lawis  \nu$. Then
  \begin{align*}
    \{\mu\in\cP(\R): \mu\mb\nu\} = \big\{\mbox{law of }\E[Y|f(Y)]: \, f\mbox{ measurable}\big\}.
  \end{align*} 
  If $\nu$ is atomless, then furthermore
    \begin{align*}
    \{\mu\in\cP(\R): \mu\mb\nu\} = \big\{\mbox{law of }\E[Y|X]: \, X\in L^{0}\big\}=\{\mu\in\cP(\R): \mu\le_{\rm cx}\nu\}
  \end{align*} 
  where $L^{0}$ is the set of random variables on the same space as~$Y$.
\end{proposition} 

\begin{proof}
  The second part follows directly from Theorem~\ref{th:refinedStrassen}. For the first part, the inclusion ``$\subseteq$'' is immediate from the definition of $\mb$. To see ``$\supseteq$'', let $\mu  \lawis Z:= \E[Y|f(Y)]$ for some measurable function~$f$. As~$Z$ is $\sigma(f(Y))$-measurable, we can write $Z=h(Y)$ for some measurable function~$h$. The tower property  of conditional expectation gives  $Z=\E[\E[Y|f(Y)]|Z] =\E[Y|Z]$. Therefore, $h(Y)=\E[Y|h(Y)]$, showing that~$h$ is the Monge map as required in the definition of $\mu\mb\nu$.
\end{proof}

\subsection{Backward deterministic martingales}
  
Theorem~\ref{thm1} gives rise to the remarkable class of backward deterministic martingales.

\begin{definition}\label{de:backwDetProc}
  A stochastic process $(X_n)_{n\in \N}$ is \emph{backward deterministic} if $(X_j)_{j=1}^n$ is $\sigma(X_n)$-measurable for all $n\in\N$.
\end{definition} 

In that case, $(X_n)_{n\in \N}$ is indeed a ``deterministic'' process if we go backward in time: the path $\{X_j,\, 1\leq j \leq n\}$ is deterministic given $X_n$. Equivalently, $\sigma(X_j)$ is non-decreasing in $n$.
As a direct consequence, a backward deterministic process $(X_n)_{n\in \N}$ is Markovian; in fact, it has perfect memory in the sense that its time-$n$ value records all its history up to time~$n$. While this may seem to be a fairly rare property, the following consequence of  Theorem~\ref{thm1} shows that the class of backward deterministic martingales is rich enough to mimic (in the sense of \cite{Gyongy.86}) any given martingale with continuous marginals. 

\begin{corollary}\label{coro:1}
Given any martingale $(Y_n)_{n\in \N}$ 
with atomless marginals, there exists a backward deterministic martingale $(X_n)_{n\in \N}$ such that $X_n\laweq Y_n$ for all $n\in \N$.
\end{corollary}

\begin{proof}
Let $\mu_n$ be the distribution of $Y_n$ for $n\in \N$. Then $\mu_n\le_{\rm cx} \mu_{n+1}$, so that Theorem~\ref{thm1} provides a sequence $\pi_n\in \mathcal{M}_M(\mu_n,\mu_{n+1})$, $n\in \N$.
Let $U_n$, $n\in \N$ be a sequence of iid  random variables uniformly distributed on $[0,1]$.
We construct the sequence $(X_n)_{n\in \N}$ inductively as follows. First, let $X_1=g(U_1)$ where $g$ is the left quantile function of $\mu_1$; then $X_1\lawis \mu_1$.  
  For $n= 2,3,\dots,$ 
  let $X_{n}$ be such that $(X_{n-1},X_{n})\lawis \pi_{n-1}$ and $X_{n}$ is measurable with respect to 
  $(U_1,\dots,U_{n})$. Such a sequence can be constructed by the inverse Rosenblatt transform; see, e.g., \citet[Theorem~1.10]{R13}.
  Then $(X_n)_{n\in\N}$ is a martingale with the marginal distributions $\mu_n$, $n\in \N$.
  Moreover, since $(X_{n-1},X_{n})\lawis \pi_{n-1}$
  and $\pi_{n-1}$ is an MMT, $X_{n-1}$ is a function of $X_{n}$ for each $n\ge 2$. Applying this repeatedly, we see that $X_j$ is a function of $X_n$ for all $j=1,\dots,n$.
\end{proof}

The celebrated mimicking theorem of \cite{Gyongy.86} shows that the marginals of a (possibly non-Markovian) It\^o process can also be generated with a Markovian It\^o process. Here, in discrete time, we provide a mimicking martingale that is even backward deterministic. Of course, the relevant input of Corollary~\ref{coro:1} is a family of distributions increasing in convex order rather than the process~$(Y_{n})$. In that sense, it is a result about ``peacocks'' in the sense of \cite{HPRY11}. To the best of our knowledge, the class of backward deterministic martingales has not been discussed in the previous literature. A deeper investigation remains for future work; we limit ourselves to the following observation.

\begin{remark}
A backward deterministic  martingale  $(X_n)_{n\in\N}$ cannot be a Gaussian process, %
except for the trivial form $(c,\dots,c,Z,Z,\cdots)$ for some $c\in\R$ and Gaussian random variable $Z$. Indeed, suppose that $(X_n)_{n\in\N}$ is a backward deterministic  martingale and a centered Gaussian process. It is clear that the variance  $\sigma_n^2 $ of $X_n$ is increasing in $n$. Moreover, for $k<n$, $\E[X_nX_k]=\E[X_k^2]=\sigma_k^2$ since $(X_n)_{n\in\N} $ is a martingale.
As the centered Gaussian distribution with a given covariance is unique, we conclude that $X_k $ cannot be a function of $X_n$ unless $\sigma_k=\sigma_n$ or $\sigma_k=0$. Hence, for some $k_{0}\in \N$, it holds that $X_k=0$ for $k<k_{0}$ and $X_{k}=X_{k_{0}}$ for $k\ge k_{0}$. At a higher level, the joint distribution of a
backward martingale is concentrated on a set of Hausdorff dimension one (contrasting that a positive definite Gaussian vector is supported on the entire space).
\end{remark}

\subsection{MMTs as minimizers of generalized optimal transport}

In this section we characterize $\pib$ through a generalized optimal transport problem.
Starting with \cite{GRST17}, transport costs involving conditional distributions have been studied under the name of \emph{generalized} or \emph{weak} optimal transport.
Such problems have found manifold applications such as the geometric inequalities of \cite{GRST17} and the Brenier--Strassen theorem of \cite{GJ20}, and have counterparts to classic concepts such as the Kantorovich duality and cyclical monotonicity established by \cite{GRST17} and \cite{BBP19}. We refer to \cite{BP22} for a recent survey.

Fix $\mu\cx\nu$ with $\nu$ atomless. It will be convenient to use random vectors $(X,Y)$ instead of joint distributions; e.g., we abuse notation and write $(X,Y)\in \Pi(\mu,\nu)$. We first note that $\pib$ naturally arises through a two-stage optimization problem. The primary optimization is to minimize $\E[\E[Y-X|X]^2]$ over $\Pi(\mu,\nu)$, and its $\argmin$ is given by $\pim$. The secondary optimization is to minimize $\E[\E[Y-X|Y]^2]$, or equivalently $\E[\var[X|Y]]$, over $\pim$; here the $\argmin$ is $\pib$. This is a symmetric variant of the \emph{barycentric optimal transport cost} introduced by \cite{GRST17}. Extending this idea, the following result represents $\pib$ as the $\argmin$ of a class of generalized optimal transport problems over~$\Pi(\mu,\nu)$.

\begin{proposition}\label{prop:opt}
  Consider $\mu\cx \nu$ with~$\nu$ atomless. For any strictly convex $f,g:\R\to\R$,
 \begin{align}\label{eq:opt}
   \pib = \argmin_{(X,Y) \in \Pi(\mu,\nu)} \E\left[f(\E[Y|X]-X)-g(\E[X|Y])\right].
 \end{align}
\end{proposition}

\begin{proof}
Recall from Theorem~\ref{thm1} that $\pib\neq\emptyset$, let $f,g:\R\to\R$ be strictly convex and $(X,Y) \in \Pi(\mu,\nu)$. Using the conditional Jensen's inequality and recalling that $\mu\cx\nu$ implies $\E[X]=\E[Y]$, 
\begin{align*} \E\left[f(\E[Y|X]-X)-g(\E[X|Y])\right] & \geq f\left(\E\left[\E[Y|X]-X\right]\right)-\E\left[\E[g(X)|Y]\right]\\&=f(\E[Y]-\E[X])-\E[g(X)]=f(0)-\E[g(X)].\end{align*}
Clearly, the right-hand side is independent of the coupling $(X,Y)\in \Pi(\mu,\nu)$. The above inequality is an equality if and  only if  $\E[Y-X|X]=0$ and $X$ is $\sigma(Y)$-measurable, or equivalently  $(X,Y)\in\pib$. \end{proof}

\begin{remark}
  For $f(x)=g(x)=x^{2}$, the generalized transport cost in~\eqref{eq:opt} is equivalent to
  \[
   \E\big[\E[Y|X]^2-\E[X|Y]^2-2\E[XY]\big].
  \]
  We note that this cost is not symmetric in $X$ and $Y$, and moreover, the term $-2\E[XY]$  is essential: one can check that $\pib$ does \emph{not} solve the problem of minimizing $ \E [\E[Y|X]^2-\E[X|Y]^2]$ unless $X$ is a constant.
\end{remark}

\section{Proofs of the main results}\label{P}

\subsection{Structure of the left-curtain transport \texorpdfstring{$\pi_{\rm lc}$}{}}\label{S}

In this subsection, we prove Proposition~\ref{p1}. Fix $\mu,\nu\in\cP(\R)$ with $\mu \le_{\rm cx} \nu$. 
 We first recall two properties of the left-curtain transport $\pi_{\rm lc}$. The first one is Theorem 1.5 of \cite{BJ16}.

\begin{lemma}[$\pi_{\rm lc}$ is left-monotone]\label{lm}
The left-curtain transport $\pi_{\rm lc}\in\pim$ satisfies $\pi_{\rm lc}(\Gamma)=1$, where $\Gamma\subseteq\R\times\R$ is a  left-monotone set; that is, whenever $(x,y^-),(x,y^+),(x',y')\in\Gamma$, it \emph{cannot} hold that
$$x<x'\quad\text{and} \quad y^-<y'<y^+.$$
Moreover, $\pi_{\rm lc}\in\pim$ is uniquely characterized by that property.
\end{lemma}

\begin{figure}[h!]  
\begin{center}

\begin{subfigure}[b]{0.3\textwidth}
\centering
\begin{tikzpicture}[scale=1.1]

\draw[gray, very thick] (-3.9,0)--(-0.9,0);
\draw[gray, very thick] (-3.9,-2)--(-0.9,-2);
\draw[very thick, ->](-3.2,0)--(-3.7,-2)node[pos=0.001,above]{$x$}node[pos=0.999,below]{$y^-$};
\draw[very thick, ->](-3.2,0)--(-1.2,-2)node[pos=0.999,below]{$y^+$};
\draw[very thick, ->](-1.7,0)--(-2.6,-2)node[pos=0.001,above]{\hspace{.17cm}$x'$}node[pos=0.999,below]{$y'$};

\end{tikzpicture}
\end{subfigure}
\caption{Forbidden configuration for left-monotonicity: the legs of a point $x'$ cannot step into the legs of another point $x$ to the left of $x'$.}\label{fig:lc}
\end{center}
\end{figure}
The second property is that, outside of $\mu$-atoms, $\pi_{\rm lc}$ is supported on the graphs of two functions (``legs'') over the first marginal (i.e., forward in time); cf.\ Corollary~1.6 of \cite{BJ16} and Theorem~1 of \cite{HN19}.

\begin{lemma}[Support of $\pi_{\rm lc}$]
There exist two functions $T_{\rm d},~T_{\rm u}:\R\to \R$ such that $\pi_{\rm lc}(R_{\rm legs}\cup R_{\rm atom})=1$, where
 \begin{enumerate}[(a)]
    \item $R_{\rm legs}$ is the union of the graphs of $T_{\rm d},~T_{\rm u}$ over the first marginal;
    \item $R_{\rm atom}=\{(x,y):\mu(\{x\})>0\}$.
\end{enumerate}
\label{su}
\end{lemma}

Define the densities 
\begin{align}d_\mu:=\frac{\d\mu}{\d(\mu+\nu) } \mbox{~~~~and~~~~} d_\nu:=\frac{\d\nu}{\d(\mu+\nu) }, \label{fg}\end{align} and denote by $\kappa_{x}(\d y)$ the disintegration of $\pi_{\rm lc}$ by $\mu$, or conditional distribution given the first marginal: $\pi_{\rm lc}(\d x,\d y)=\mu(\d x)\otimes\kappa_{x}(\d y)$.

\begin{lemma}\label{le:needSpace}
We have $d_\mu\leq d_\nu\ \mu$-a.e.\ on $\{x\in\R: \kappa_{x}=\delta_{x}\}$.
\end{lemma}
\begin{proof}
Define $A=\{x\in\R:\kappa_{x}=\delta_x\}\cap\{x\in\R:d_\mu(x)>d_\nu(x)\}$. Assuming $\mu(A)>0$, we find 
\begin{align*}\mu(A)&=\int_{A}d_\mu\,\d(\mu+\nu) >\int_{ A}d_\nu\,\d(\mu+\nu) =\nu(A) =\int\kappa_x(A)\mu(\d x) =\int\delta_x(A)\mu(\d x) =\mu(A),
\end{align*}
a contradiction.
\end{proof}

\begin{proof}[Proof of Proposition~\ref{p1}.]
We first detail the proof for the second assertion, namely that
\begin{center}
$\pi_{\rm lc}$ is supported on $S_{\rm rg}\cup S_{\rm atom}$ if $d_\mu\geq d_\nu\ \mu$-a.e.
\end{center}
and  $\pi_{\rm lc}\in\pib$ if in addition $\nu$ is atomless. 

\emph{Step~1.}  We have the martingale property $\int_\R y\,\kappa_{x}(\d y)=x$ for $\mu$-a.e.~$x$. Then by Lemma~\ref{su}, for $\mu$-a.e.~$x$ with $\mu(\{x\})=0$, either $\kappa_{x}=\delta_x$ or $\kappa_{x}$ is supported on two points $T_{\rm d}(x)<x<T_{\rm u}(x)$. Moreover, if $\mu(\{x\})>0$, then either $\kappa_{x}(\{x\})=0$ or $x$ belongs to the set $A_{\nu}=\{y\in\R: \nu(\{y\})>0\}$ of atoms of~$\nu$.  In view of Lemma~\ref{le:needSpace} and $d_\mu\geq d_\nu\ \mu$-a.e., we conclude that 
$$
  \{x\in\R:\kappa_{x}(\{x\})>0\}= \{x\in\R:\kappa_{x}=\delta_x\}\subseteq\{x\in \R :d_\mu(x)=d_\nu(x)\} \quad\mu\mbox{-a.e.\ outside }A_{\nu}.
$$
In summary, $\pi_{\rm lc}$ is the identity transport on $S:=\{x\in\R:\kappa_{x}(\{x\})>0\}\setminus A_{\nu}$ and has the backward Monge property on~$S$. Thus, we may without loss of generality ``remove'' $\mu|_{S}$ from the two marginals and assume that $\kappa_{x}(\{x\})=0$ $\mu$-a.e.\ outside~$A_{\nu}$ for the remainder of the proof.

\emph{Step~2.} Let $\Gamma$ be the left-monotone set provided by Lemma~\ref{lm}. By taking intersection, we may assume that $\Gamma\subseteq \supp(\mu)\times\supp(\nu)$ and $\Gamma\subseteq R_{\rm legs}\cup R_{\rm atom}$, where $\supp(\cdot)$ denotes topological support. By Step~1, we may further assume $(\Gamma\setminus S_{\rm atom})\cap\{(x,x):x\in\R\}=\emptyset$. Suppose that $x<x'$ are two points being transported to the same point $y\notin A_{\nu}$, or more precisely, that the pairs $(x,y),(x',y)\in\Gamma\setminus S_{\rm atom}$, and in particular $y\not\in\{x,x'\}$. Then there are three possible cases (see Figure \ref{fig:2}):
 \begin{enumerate}[(a)]
    \item If $x<y<x'$, then $\mu((x,y))=0$ (here, $(x,y)$ refers to an interval instead of a pair). Indeed, if $x_*\in(x,y)$, then  by Lemma~\ref{lm}, its right ``leg'' must lie on $y$ because otherwise the left leg of $x'$ ``steps into'' the legs of $x_*$. Since $\nu(\{y\})=0$, $\mu((x,y))=0$.
    \item If $y<x<x'$, denote by $y'$ the right leg of $x$. Then by Lemma~\ref{lm}, the left leg of any $x_*\in(x,\min\{y',x'\}]$ cannot lie to the right of $y$, to avoid stepping into the legs of $x$, and not to the left of $y$ because otherwise the left leg of $x'$ steps into the legs of $x_*$. Thus the left leg of $x_*$ must lie on $y$, implying that  $\mu((x,\min\{y',x'\}))=0$.
    
    \item If $x<x'<y$, consider $x_*\in(x,x')$. Then by Lemma~\ref{lm}, the right leg of $x_*$ cannot lie to the left of $y$,  to avoid stepping into the legs of $x$, and not to the right of $y$, because otherwise the right leg of $x'$ steps into the legs of $x_*$. This shows that the right leg of $x_*$ must lie on $y$, and thus    $\mu((x,x'))=0.$
\end{enumerate}

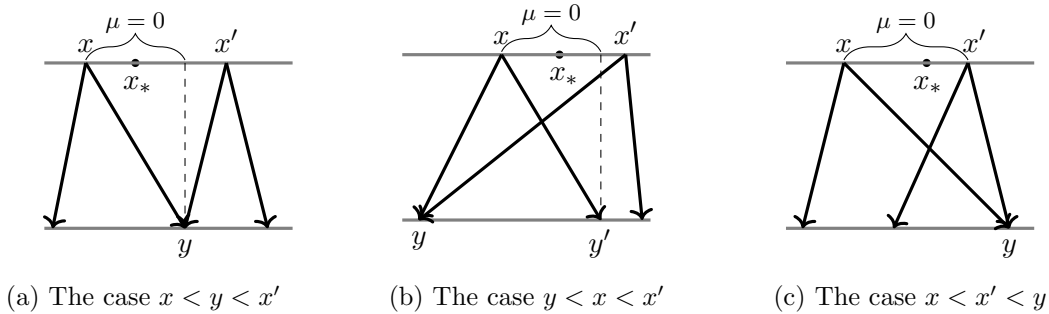
\begin{figure}[h!]  
\begin{center}
\begin{subfigure}[b]{0.3\textwidth}
\centering
\captionsetup{justification=centering}
\begin{tikzpicture} [scale=1.1]
\node[circle,fill,scale=0.3] at (-2.8,0){};
\draw[gray, very thick] (-3.9,0)--(-0.9,0);
\draw[gray, very thick] (-3.9,-2)--(-0.9,-2);
\draw[very thick, ->](-3.4,0)--(-2.2,-2)node[pos=0.001,above]{$x$};
\draw[very thick, ->](-3.4,0)--(-3.8,-2)node[pos=0.25,above]{\hspace{1.6cm}$x_*$};
\draw[very thick, ->](-1.7,0)--(-2.2,-2)node[pos=0.999,below]{$y$};
\draw[very thick, ->](-1.7,0)--(-1.2,-2)node[pos=0.001,above]{$x'$};
\draw [decorate,decoration={brace,amplitude=10pt},xshift=0pt,yshift=1pt]
(-3.4,0) -- (-2.2,0) node [black,midway,yshift=0.5cm] 
{\footnotesize $\mu=0$};
\draw[dashed](-2.2,0)--(-2.2,-2);
\end{tikzpicture}
\caption{The case $x<y<x'$}
\end{subfigure} 
\begin{subfigure}[b]{0.3\textwidth}
\centering
\begin{tikzpicture}[scale=1.1]
\node[circle,fill,scale=0.3] at (-2,0){};
\draw[gray, very thick] (-3.9,0)--(-0.9,0);
\draw[gray, very thick] (-3.9,-2)--(-0.9,-2);
\draw[very thick, ->](-2.7,0)--(-3.7,-2)node[pos=0.001,above]{$x$};
\draw[very thick, ->](-2.7,0)--(-1.5,-2)node[pos=0.25,above]{\hspace{1cm}$x_*$}node[pos=0.99,below]{$y'$};
\draw[very thick, ->](-1.2,0)--(-3.7,-2)node[pos=0.999,below]{$y$};
\draw[very thick, ->](-1.2,0)--(-1,-2)node[pos=0.001,above]{$x'$};
\draw [decorate,decoration={brace,amplitude=10pt},xshift=0pt,yshift=1pt]
(-2.7,0) -- (-1.5,0) node [black,midway,yshift=0.5cm] 
{\footnotesize $\mu=0$};
\draw[dashed](-1.5,0)--(-1.5,-2);
\end{tikzpicture}
\caption{The case $y<x<x'$}
\end{subfigure}
\begin{subfigure}[b]{0.3\textwidth}
\centering
\begin{tikzpicture}[scale=1.1]
\node[circle,fill,scale=0.3] at (-2.2,0){};
\draw[gray, very thick] (-3.9,0)--(-0.9,0);
\draw[gray, very thick] (-3.9,-2)--(-0.9,-2);
\draw[very thick, ->](-3.2,0)--(-3.7,-2)node[pos=0.001,above]{$x$};
\draw[very thick, ->](-3.2,0)--(-1.2,-2)node[pos=0.25,above]{\hspace{1.13cm}$x_*$};
\draw[very thick, ->](-1.7,0)--(-1.2,-2)node[pos=0.999,below]{$y$};
\draw[very thick, ->](-1.7,0)--(-2.6,-2)node[pos=0.001,above]{\hspace{.17cm}$x'$};
\draw [decorate,decoration={brace,amplitude=10pt},xshift=0pt,yshift=1pt]
(-3.2,0) -- (-1.7,0) node [black,midway,yshift=0.5cm] 
{\footnotesize $\mu=0$};
\end{tikzpicture}
\caption{The case $x<x'<y$}
\end{subfigure}
\caption{Illustration of the three cases}\label{fig:2}
\end{center}
\end{figure}

As $\supp(\mu)$ is closed, its complement can be written as a countable disjoint union of open intervals. 
We have shown that  each non-injective pair of $(x,y),(x',y)\in\Gamma\setminus S_{\rm atom}$ with $x\neq x'$  corresponds to an endpoint of one of the open intervals, and the map from the collection of all intervals to the collection of $y$ values is at most one-to-two (since there are at most two legs). Thus, there are at most countably many such points~$y$, and as $\nu$ is atomless outside~$A_{\nu}$, it follows that these points are $\nu$-negligible. In summary, we have shown that $\pi_{\rm lc}$ is supported on the union of the (reverse) graph $S_{\rm rg}$ of a function $h:\R\to\R$ and $S_{\rm atom}$.

It remains to see that $h$ can be chosen to be measurable, and that $\pi_{\rm lc}=(h,\Id_{\R})_{\#}\nu$ when~$\nu$ is atomless. In the latter case, the mere fact that $\pi_{\rm lc}$ is concentrated on the graph of~$h$ already implies that~$h$ is $\nu$-measurable and 
$\pi_{\rm lc}=(h,\Id_{\R})_{\#}\nu$; see \cite[Lemma~3.1]{AhmadKimMcCann.11} for a detailed argument exploiting the inner regularity of Borel measures. Redefining~$h$ on a $\nu$-nullset then gives the desired Borel measurable function. In the case with atoms, we can apply the same  lemma to the restriction $\pi'$ of $\pi_{\rm lc}$ to the Borel set $\R^{2}\setminus S_{\rm atom}$. The lemma then yields that $h$ is $\nu'$-measurable where~$\nu'$ is the second marginal of~$\pi'$, and we can again extract a Borel version. This completes the proof of the second assertion in Proposition~\ref{p1}.

The proof of the first assertion, namely that $\pi_{\rm lc}$ is supported on $S_{\rm rg}\cup S_{\rm diag}\cup S_{\rm atom}$, is similar to Step~2 above (but simpler): we now argue on the left-monotone set $\Gamma \setminus (S_{\rm diag}\cup S_{\rm atom})$.
\end{proof}

\begin{remark}[When is~$\pi_{\rm lc}$ Monge?] 
While not directly required for our main results, it seems natural to ask when~$\pi_{\rm lc}$ has the (reverse) Monge property. In the following discussion, we assume that~$\nu$ is atomless.
First of all, we note that the converse of Proposition~\ref{p1} is false:  $\pi_{\rm lc}\in \pib$ does not imply that 
$d_\mu\geq d_\nu\ \mu$-a.e. This can be seen by choosing the black density in Figure~\ref{fig:3too}   small enough.

Recall that $\pi_{\rm lc}$ is supported on the union of the (forward) graphs of $T_{\rm d}$ and $T_{\rm u}$.  It follows from  Proposition~\ref{p1} that $\pi_{\rm lc}$ is Monge if and only if
 $d_{\mu}=d_{\nu}$ $\mu$-a.e.~on  the set  $\{x\in\R:T_{\rm d}(x)=T_{\rm u}(x)\}$ where the two legs of $\pi_{\rm lc}$ coincide. 
Under additional regularity assumptions, the main results of \cite{HT16} imply (somewhat convoluted) equivalent conditions for this that can be stated in terms of the primitives~$\mu$ and~$\nu$.
To see the basic complication, consider $x\in\R$ with $d_\mu(x)\leq d_\nu(x)$. It is possible that $T_{\rm d}(x)=T_{\rm u}(x)$, i.e., the two legs coincide, while it is also possible that the $\nu$-mass at $x$ already lies in the shadow of $\mu|_{(-\infty,y]}$ for some $y<x$, making the legs separate instead, as shown in Figure~\ref{fig:3too}. 
In the proof of Theorem~\ref{thm1} below, we circumvent these issues by using the tractable sufficient condition $d_\mu\geq d_\nu$ and guaranteeing it through the decomposition into bars.
\begin{figure}[h!] 
\begin{center}
 \scalebox{1}{\begin{subfigure}[b]{0.4\textwidth}
\centering
\captionsetup{justification=centering}
\begin{tikzpicture}[scale=1.9]
\draw[blue, very thick] (-4.5,0)--(-1.7,0)  node[pos=1,right, black] {$\mu$};
\draw[blue, very thick] (-4.5,-1.2)--(-1.7,-1.2)  node[pos=1,right, black] {$\nu$};
\filldraw[black!30, very thick] (-3.6,0) rectangle (-3.4,1);
\filldraw[black!60, very thick] (-2.8,0) rectangle (-2.6,1);
\filldraw[black!100,very thick] (-3.25,0) rectangle (-2.95,0.133);

\filldraw[black!30, very thick] (-4,-1.2) rectangle (-3,-1);
\filldraw[black!100, very thick] (-3,-1.2) rectangle (-2.84,-1);
\filldraw[black!100, very thick] (-4.04,-1.2) rectangle (-4,-1);
\filldraw[black!60,very thick] (-4.2,-1.2) rectangle (-4.04,-1);
\filldraw[black!60,very thick] (-2.84,-1.2) rectangle (-2,-1);

\draw[very thick, ->](-3.1,-0.42)--(-3.1,-0.85);

 \draw [decorate,decoration={brace,mirror,amplitude=7pt},xshift=0pt]
(-3.27,0) -- (-2.93,0) node [black,midway,yshift=-0.5cm]  {$d_\mu< d_\nu$};
\end{tikzpicture}
\end{subfigure}}

\caption{The left-curtain transport $\pi_{\rm lc}$ is not the identity on $\{x\in\R:d_\mu(x)< d_\nu(x)\}$}\label{fig:3too}
\end{center}\end{figure}
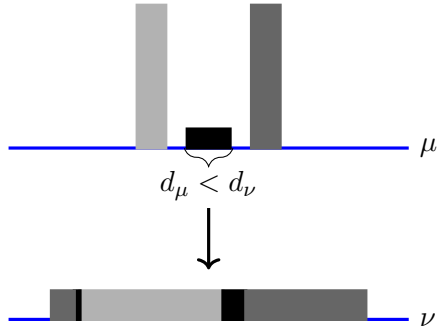
\end{remark} 
\subsection{Proof of Theorem~\ref{thm1}}

We follow the notation of Section~\ref{S} but consider possibly unnormalized finite measures $\mu,\nu$ on~$\R$ as the following auxiliary results will be applied to sub-measures of the given marginals. We denote the barycenter by $\bary(\mu):=\int_\R x\,\mu(\d x)/\mu(\R)$. 

\begin{lemma}\label{0}
If $\mu(\R)=\nu(\R)>0$, then 
$\mu\left(\left\{d_{\mu}\geq d_{\nu}\right\}\right)>0.$
\end{lemma}
\begin{proof}
Suppose $\mu\left(\left\{d_{\mu}\geq d_{\nu}\right\}\right)=0$, then also $\nu\left(\left\{d_{\mu}\geq d_{\nu}\right\}\right)=0$. Thus 
$\mu(\R)=\mu\left(\left\{d_{\mu}< d_{\nu}\right\}\right)
<\nu\left(\left\{d_{\mu}< d_{\nu}\right\}\right)=\nu(\R)$,
contradicting our assumption.
\end{proof}

Two properties of shadows will be used repeatedly. The first is due to \cite[Theorem 4.8]{BJ16}.
\begin{lemma}[Associativity of shadows]\label{ass}
Suppose that $\mu=\mu_1+\mu_2\e\nu$. Then $\mu_2\e\nu-S^\nu(\mu_1)$ and $S^\nu(\mu)=S^\nu(\mu_1)+S^{\nu-S^\nu(\mu_1)}(\mu_2)$.
\end{lemma}

 The second can be found in \cite[Example 4.7]{BJ16}.

\begin{lemma}\label{atom} When $\nu$ is atomless, the shadow of an atom of $\mu$ is $\nu$ restricted to an interval.
\end{lemma}

The following significantly generalizes Lemma~\ref{atom} by using Proposition~\ref{p1}.

\begin{lemma}\label{ms} Consider $\mu\e\nu$ with $d_{\mu}\geq d_{\nu}$ $\mu$-a.e. Then $S^\nu(\mu)$ and $\nu-S^\nu(\mu)$ are mutually singular outside of $\{y\in\R:\nu(\{y\})>0\}$.
\end{lemma}

\begin{proof}
In case $\mu(\R)=\nu(\R)$, it must hold that $S^\nu(\mu)=\nu$ and the conclusion is vacuously true. Thus we may assume $\mu(\R)<\nu(\R)$. Since $\mu\e\nu$, we may add to $\mu$ a Dirac mass to get a  measure dominated by $\nu$ in convex order: taking $\lambda=\nu(\R)-\mu(\R)$ and $m=\lambda^{-1}(\nu(\R)\bary(\nu)-\mu(\R)\bary(\mu))$ yields that  $\mu+\lambda\delta_m\cx\nu$.
 Applying Proposition~\ref{p1} to the measures $\mu+\lambda\delta_m$ and $ \nu$ yields that the left-curtain transport from $\mu+\lambda\delta_m$ to $ \nu$ is Monge outside the set $A:=\{y\in\R:\nu(\{y\})>0\}$ of atoms of~$\nu$.  Since the left-curtain transport sends $\mu|_{(-\infty,m]}$ to its shadow $S^\nu(\mu|_{(-\infty,m]})$, we deduce that $S^\nu(\mu|_{(-\infty,m]})$ and $\nu':=\nu-S^\nu(\mu|_{(-\infty,m]})$ are mutually singular outside of~$A$.  Note that
$$\frac{\d(\mu|_{(m,\infty)}+\lambda\delta_m)}{\d(\mu|_{(m,\infty)}+\lambda\delta_m+\nu')}\geq \frac{1}{2},\quad (\mu|_{(m,\infty)}+\lambda\delta_m)\text{-a.e}.$$
By a symmetrical argument using Proposition~\ref{p1}, the right-curtain transport from $\mu|_{(m,\infty)}+\lambda\delta_m$ to $\nu'$ is backward Monge outside of~$A$ and sends $\mu|_{(m,\infty)}$ to $S^{\nu'}(\mu|_{(m,\infty)})$, and thus $S^{\nu'}(\mu|_{(m,\infty)})$ and $\nu'-S^{\nu'}(\mu|_{(m,\infty)})$ are mutually singular. By Lemma~\ref{ass}, it holds that
$S^\nu(\mu)=S^\nu(\mu|_{(-\infty,m]})+S^{\nu'}(\mu|_{(m,\infty)}).$ Therefore, $S^\nu(\mu)$ and $\nu-S^\nu(\mu)$ are mutually singular outside of~$A$.  
\end{proof}

We can now construct the barcode transport.

\begin{proof}[Proof of Theorem~\ref{thm1}] Given $\mu,\nu\in\mathcal P(\R)$ with $\mu\cx\nu$, we let $(d^{(0)}_\mu,d^{(0)}_\nu):=(d_\mu,d_\nu)$ be defined as in~\eqref{fg}. Consider $A_0:=\{d^{(0)}_\mu\geq d^{(0)}_\nu\}$. We transport $\mu|_{A_0}$ to $S^\nu(\mu|_{A_0})$ using the left-curtain coupling, which is Monge outside the set $\{y\in\R:\nu(\{y\})>0\}$ of atoms of~$\nu$ by Proposition~\ref{p1}. (In Figure~\ref{fig:1}\,(a), this corresponds to the light-gray area in the center.) Define the remaining measures 
$$\mu_1:=\mu-\mu|_{A_0}, \qquad \nu_1:=\nu-S^\nu(\mu|_{A_0}),$$
so that $\mu_1\cx\nu_1$. We continue recursively: given $n\in\N$ and measures $\mu_n\cx\nu_n$, we define the densities $d^{(n)}_\mu,d^{(n)}_\nu$ of $\mu_n,\nu_n$ with respect to $(\mu+\nu)$ and $A_n:=\{d^{(n)}_\mu\geq d^{(n)}_\nu\}$, as well as
$$\mu_{n+1}:=\mu_n-\mu_n|_{A_n}, \qquad \nu_{n+1}:=\nu_n-S^{\nu_n}(\mu_n|_{A_n}).$$Let also $\pi_n\in \mathcal{M}(\mu_n|_{A_n},S^{\nu_n}(\mu_n|_{A_n}))$ be the left-curtain transport, which is again Monge outside the atoms of~$\nu$ by Proposition~\ref{p1}. By construction, the measures $\{\mu_n-\mu_{n+1}\}$ are mutually singular, and by Lemma~\ref{ms}, the measures $\{\nu_n-\nu_{n+1}\}$ are mutually singular outside of $\{y\in\R:\nu(\{y\})>0\}$.

Again by construction, we have that $d^{(n)}_\mu,d^{(n)}_\nu$ are decreasing sequences of functions $(\mu+\nu)$-a.e. Denote their limits $d^{(\infty)}_\mu,d^{(\infty)}_\nu$ respectively. Let $x\in\R$ belong to the $(\mu+\nu)$-a.e.~set where $d^{(n)}_\mu,d^{(n)}_\nu$ are decreasing and such that $d^{(\infty)}_\mu(x)\geq d^{(\infty)}_\nu(x)$. Then by mutual singularity of $\{\nu_n-\nu_{n+1}\}$, we have $d^{(n)}_\nu(x)\in\{d^{(0)}_\nu(x),0\}$ for all $n$. There are two possible cases:
 \begin{enumerate}[(a)]
    \item Suppose that there is a finite $n$ such that $d^{(n)}_\nu(x)=0$. Then $d^{(\infty)}_\nu(x)=0$ and $d^{(n)}_\mu(x)\geq d^{(n)}_\nu(x)$. This means that the $\mu$-mass at $x$ must be transported at step $n+1$ or earlier, giving that $d^{(n+1)}_\mu(x)=0$.
    \item Suppose that $d^{(n)}_\nu(x)=d^{(0)}_\nu(x)$ for all $n$. Then $d^{(0)}_\mu(x)\geq d^{(\infty)}_\mu(x)\geq d^{(\infty)}_\nu(x)=d^{(0)}_\nu(x)$. By construction, the $\mu$-mass at $x$ must be transported in the first step, so that $d^{(1)}_\mu(x)=0$.
\end{enumerate}
It follows that $d^{(\infty)}_\mu(x)=0$. Therefore, $\mu_\infty$ is the zero measure by Lemma~\ref{0}, and so is $\nu_\infty$ since $\mu_n(\R)=\nu_n(\R)$ by construction. Since outside  of $\{y\in\R:\nu(\{y\})>0\}$, each transport  $\pi_n\in\mathcal{M}(\mu_n-\mu_{n+1},\nu_n-\nu_{n+1})$ is Monge  and the measures $\{\nu_n-\nu_{n+1}\}$ are mutually singular, aggregating these transports yields a transport from~$\mu$ to~$\nu$ that is Monge outside that set.
\end{proof}

We remark that, by construction, the barcode transport belongs to the broad class of shadow couplings introduced by \cite{BJ21}. While our construction uses the left-curtain transport for its relatively simple behavior, this is certainly not the only possible choice.

\begin{remark}\label{nu}
Even if the left-curtain transport is an MMT for two given marginals, our construction may result in a different transport; see Figure~\ref{fig:3} for an example.
  
\begin{figure}[ht!]
\begin{center}
\begin{subfigure}[b]{0.49\textwidth}
\centering
\captionsetup{justification=centering}
\scalebox{1}{\begin{tikzpicture}[scale=1.9]
\draw[blue, very thick] (-4.5,0)--(-1.7,0) node[pos=1,right, black] {$\mu$};
\draw[blue, very thick] (-4.5,-1.2)--(-1.7,-1.2) node[pos=1,right, black] {$\nu$};
\filldraw[black!30, very thick] (-3.6,0) rectangle (-3.4,1);
\filldraw[black!60, very thick] (-2.8,0) rectangle (-2.6,1);
\filldraw[black!100,very thick] (-3.25,0) rectangle (-2.95,0.133);

\filldraw[black!30, very thick] (-4,-1.2) rectangle (-3,-1);
\filldraw[black!100, very thick] (-3,-1.2) rectangle (-2.84,-1);
\filldraw[black!100, very thick] (-4.04,-1.2) rectangle (-4,-1);
\filldraw[black!60,very thick] (-4.2,-1.2) rectangle (-4.04,-1);
\filldraw[black!60,very thick] (-2.84,-1.2) rectangle (-2,-1);

\draw[very thick, ->](-3.1,-0.2)--(-3.1,-0.8);
\end{tikzpicture}}
\caption{Left-curtain transport}
\end{subfigure} 
\begin{subfigure}[b]{0.49\textwidth}
\centering
\scalebox{1}{\begin{tikzpicture}[scale=1.9]
\draw[blue, very thick] (-4.5,0)--(-1.7,0) node[pos=1,right, black] {$\mu$};
\draw[blue, very thick] (-4.5,-1.2)--(-1.7,-1.2) node[pos=1,right, black] {$\nu$};
\filldraw[black!30, very thick] (-3.6,0) rectangle (-3.4,1);
\filldraw[black!60, very thick] (-2.8,0) rectangle (-2.6,1);
\filldraw[black!100,very thick] (-3.25,0) rectangle (-2.95,0.133);

\filldraw[black!30, very thick] (-4,-1.2) rectangle (-3,-1);
\filldraw[black!100, very thick] (-2.1,-1.2) rectangle (-2,-1);
\filldraw[black!100, very thick] (-4.2,-1.2) rectangle (-4.1,-1);
\filldraw[black!60,very thick] (-3,-1.2) rectangle (-2.1,-1);
\filldraw[black!60,very thick] (-4.1,-1.2) rectangle (-4,-1);

\draw[very thick, ->](-3.1,-0.2)--(-3.1,-0.8);
\end{tikzpicture}}
\caption{Barcode transport}
\end{subfigure}
\caption{Left-curtain and barcode transport are MMTs, yet do not coincide}\label{fig:3}
\end{center}\end{figure}
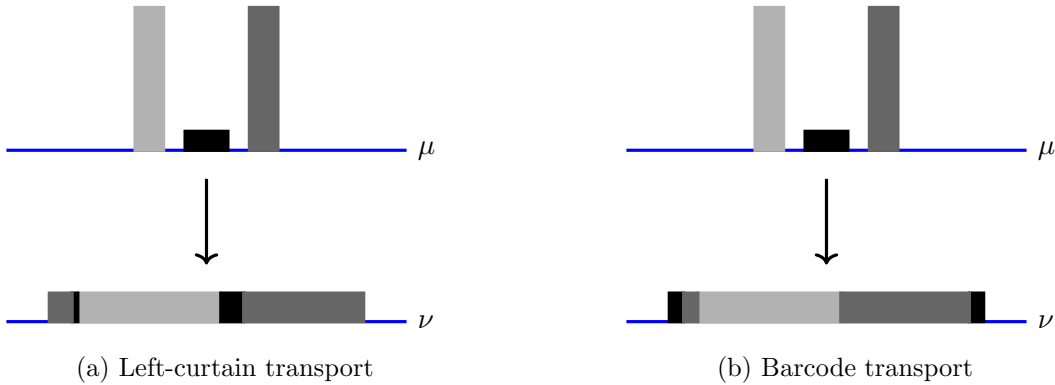

\end{remark}
  
\subsection{Proof of Theorem~\ref{th:dense}}\label{se:densityProof}

Let $\mu,\nu\in\mathcal P(\R)$ where $\nu$ is atomless. For $p\in [1,\infty]$, we denote by $W_p$ the $p$-Wasserstein distance of measures on either $\R$ or $\R^2$ equipped with the Euclidean metric. While the two assertions of Theorem~\ref{th:dense} will be proved independently, the proof for discrete~$\mu$ is presented first as it is much simpler yet contains some of the basic ideas for both cases.

 \begin{lemma}\label{lem:june-1}
 Let $\nu\in\mathcal P(\R)$ be atomless. Given any decomposition $\nu=\sum_{i=1}^\infty \nu_i$ of $\nu$,   there exist mutually singular $\tilde \nu_i, $ ${i\in \N}$ such that 
$\nu=\sum_{i=1}^\infty  \tilde \nu_i$ and $\nu_1\le_{\rm cx} \tilde \nu_1$ and 
 $\bary(\nu_i)=\bary(\tilde \nu_i)$ for $i\geq 2$.
 \end{lemma} 
 
 \begin{proof}
 Define $\tilde \mu_i= \nu_i(\R) \delta_{\bary(\nu_i)}$, $i\in \N$. 
 Note that $ \sum_{i=1}^\infty \tilde \mu_i\cx \nu_1+\sum_{i=2}^\infty \tilde \mu_i  \le_{\rm cx}\sum_{i=1}^\infty \nu_i = \nu$. 
 We %
consider the shadow  $ \nu_0:=S^{\nu}(\sum_{i=2}^\infty  \tilde \mu_i)$ and set $\tilde \nu_1=\nu - \nu_0$. By Lemma~\ref{ass} and Lemma~\ref{atom}, $\nu_1\cx\tilde\nu_1$ and $\tilde\nu_1$ is mutually singular with $\nu_0$. Roughly speaking, $\tilde\nu_1$ is the \emph{largest} possible image of~$\nu_{1}$ under a martingale transport, in the sense of the convex order.

Next, we apply a shadow coupling from $ \sum_{i=2}^\infty  \tilde \mu_i $ 
 to $ \nu_0$, processing these atoms in the order $i=2,3,\dots$. More precisely,
we let 
 $\tilde \nu_2:=S^{\nu_0}(\tilde \mu_2)$ and  $\tilde \nu_i:=S^{\nu_0-\sum_{j=2}^i \tilde  \mu_j}(\tilde \mu_i)$ for $i\ge 3$. 
 By construction and Lemma~\ref{atom}, these shadows $\tilde \nu_i$, $i\ge 2$, are mutually singular. As sub-measures of~$\nu_{0}$, they are also mutually singular with $\tilde \nu_1$. The other assertions are clear.
 \end{proof}

 \begin{proof}[Proof of Theorem~\ref{th:dense} for discrete $\mu$] Fix $\pi\in\pim$ and $\ee>0$; we construct $\pi_\ee\in\pib$ with $W_\infty(\pi,\pi_\ee)\leq \ee$. Partition $\R$ into intervals $\{I_\ell\}_{\ell\in\N}$ of length $\ee$  and write $\nu=\sum_{\ell=1}^\infty\nu|_{I_\ell}$. Decompose the discrete measure $\mu$ into its atoms, $\mu=\sum_{k=1}^\infty \mu_k$. Then, decompose $\nu|_{I_\ell}=\sum_{k=1}^\infty\nu_{k,\ell}$ where $\nu_{k,\ell}$ is the image of $\mu_k$ under $\pi$ restricted to~$I_\ell$. For each $\ell$, apply Lemma~\ref{lem:june-1} to the decomposition $\nu|_{I_\ell}=\sum_{k=1}^\infty\nu_{k,\ell}$, yielding measures $\{\mu_{k,\ell}\}_{k,\ell\in\N}$ and $\{\tilde\nu_{k,\ell}\}_{k,\ell\in\N}$ such that $\mu_k=\sum_{\ell=1}^\infty \mu_{k,\ell}$ and $\nu|_{I_\ell}=\sum_{k=1}^\infty \tilde\nu_{k,\ell}$ and  $\mu_{k,\ell}\cx\tilde\nu_{k,\ell}$ and  $\{\tilde\nu_{k,\ell}\}_{k,\ell\in\N}$ are mutually singular. Moreover, $W_\infty(\tilde\nu_{k,\ell},\nu_{k,\ell})\leq\ee$ for all $k,\ell$ by construction.
 Consider the transport $\pi_\ee\in\pim$ that sends each atom $\mu_{k,\ell}$ to $\tilde\nu_{k,\ell}$. Then $\pi_\ee\in\pib$ since $\{\tilde\nu_{k,\ell}\}_{k,\ell}$ are mutually singular, and $W_\infty(\pi,\pi_\ee)\leq \ee$ since $W_\infty(\tilde\nu_{k,\ell},\nu_{k,\ell})\leq\ee$.
  \end{proof}

Before entering the technical details of the proof of Theorem~\ref{th:dense} for general~$\mu$, let us try to sketch the main ideas. Similarly as in the discrete case above, we want to partition the supports of~$\mu$ and~$\nu$ into small enough intervals $\{J_k\},\{I_\ell\}$ and define $\nu_k$ to be the image of $\mu|_{J_k}$ under the given transport $\pi_{0}\in\pim$ to be approximated. Using barcodes, we would then approximate the measures $\nu_{k}|_{I_\ell}$ within the set $I_\ell$ for each $\ell$, meaning that we find mutually singular $\{\hat\nu_{k,\ell}\}$ such that $\sum_{k}{\nu_{k}|_{I_\ell}}=\sum_k\hat\nu_{k,\ell}$ for each $\ell$. This idea does  not carry through directly, because these rearrangements may destroy vital convex order properties. Instead, we perform yet another approximation to create some ``wiggle room'' in the convex order. Rather than directly approximating the given coupling $\pi_{0}$, we approximate $\tilde \pi_0= (1-\lambda)\pi_0+ \lambda \pi_3$ for small $\lambda$ and a particular martingale transport $\pi_{3}\in\pim$ with a tailored transport kernel based on a carefully chosen Rademacher noise. Roughly speaking, adding the noise yields a locally uniform lower bound on the dispersion of the transport kernels. 

It will be important to quantify how far two marginals are separated from one another in the convex order---specifically, how large a perturbation (in $W_{\infty}$) can be applied without violating the order. To that end, the characterization of the convex order by potential functions is useful. 
  The potential function $u_\mu:\R\to \R$ of $\mu$ is defined as $x\mapsto \int_\R  |y-x| \d \mu(y)$. This function is convex and Lipschitz. 
If $\mu$ and $\nu$  have the same mass and barycenter, then $u_\mu \le u_\nu$ if and only if $\mu \le_{\rm cx} \nu$; see  \citet[Theorem 3.A.2]{SS07}. The difference $u_\nu(x)-u_\mu(x)$ will be used as a local measure of separation between the marginals.

\begin{lemma}\label{lem:irred}
Without loss of generality, we may assume that $I:=\{u_{\mu}<u_{\nu}\}$ is an (open) interval and that $\mu(I)=\nu(I)=1$. In particular, $\mu(\{u_\mu = u_\nu\})=0$.
\end{lemma}

\begin{proof}
  Consider the decomposition $\mu= \sum_{i\geq0}\mu_{i}$ and $\nu=\sum_{i\geq0}\nu_{i}$ of $(\mu,\nu)$ into the so-called irreducible components; cf.~\cite[Theorem~A.4]{BJ16}. Here $(\mu_{i},\nu_{i})$ are in convex order and any $\pi\in\pim$ transports~$\mu_{i}$ to~$\nu_{i}$. Moreover,  $\mu_{0}=\nu_{0}$ are such that any $\pi\in\pim$ transports~$\mu_{0}$ to~$\nu_{0}$ via the identity transport. Finally, $(\mu_{i})_{i\geq1}$ are supported on the disjoint intervals $\{u_{\mu_{i}}<u_{\nu_{i}}\}$ and $\mu_{0}$ is supported on the complement of their union. The same holds for $(\nu_{i})_{i\geq0}$, as follows from \cite[Lemma~A.6]{BJ16}: while in general $\nu_{i}$ may place mass at the endpoints of its interval, that is not the case here as~$\nu$ is atomless. It follows that any $\pi\in\pim$ is Monge on the complement of the intervals (since the only transport there is the identity), and if the denseness result of Theorem~\ref{th:dense} holds for each $(\mu_{i},\nu_{i})$ with $i\geq1$, then aggregating yields the desired theorem for~$(\mu,\nu)$.
\end{proof}

In the remainder of the proof, we assume that the condition of Lemma~\ref{lem:irred} holds.

\begin{lemma}\label{lem:june-2}
We have 
$\lim_{\delta \downarrow 0}\mu(A_\delta)= 0$ for $A_\delta :=[-\delta,\delta]+\{x\in \R:0\leq u_\nu(x)-u_\mu(x)<\delta\}$.
\end{lemma}

\begin{proof}
  The sets $A_\delta$ are decreasing and $\cap_{\delta>0}A_\delta=\{u_\mu = u_\nu\}$ which is $\mu$-null by our assumption.
\end{proof}

The next lemma quantifies how much ``wiggle room" of convex order the Rademacher noise introduces into a distribution. We denote by $\mathrm{Rade}$ the Rademacher distribution, or uniform on $\{-1,+1\}$.

\begin{lemma}\label{l1}
Fix $x_0\in\R$, $\lambda\in(0,1]$, and $\ee>0$. Let $\mu_1$ be a probability measure with mean~$x_0$  such that $\mu_1([x_0-\lambda\ee/6,x_0+\lambda\ee/6])=1$, and $\mu_2$ be the distribution of $X_1+\ee B\xi $ where $X_1\lawis \mu_1$,   $B\lawis \mathrm{Bernoulli}(\lambda)$  and $\xi\lawis \mathrm{Rade}$ are independent. Suppose that $\mu_3$ and $\mu_4$ are probability measures with the same mean $x_0$ such that $\mu_2\cx\mu_3$ and $W_\infty(\mu_3,\mu_4)\leq \lambda\ee/6$. Then $\mu_1\cx\mu_4$.
\end{lemma}

\begin{proof}
  We first claim that there exists $\mu_3'$ with mean $x_0$ such that  $W_\infty(\mu_2,\mu_3')<\lambda\ee/6$ and $\mu_3'\cx\mu_4$. Using the  disintegration theorem, we may write kernels $\kappa^{(2)}_x$ and $\kappa^{(3)}_x$
 that transport $\mu_2$ to $\mu_3$ and $\mu_3$ to $\mu_4$ respectively, such that the mean of $\kappa^{(2)}_x$ is $x$ (i.e., $\kappa^{(2)}_x$ is an  MT) and $\kappa^{(3)}_x$ is concentrated in $[x-\lambda\ee/6,x+\lambda\ee/6]$ for each $x\in \R$. 
 Denote by $x^*$ the mean of the measure $\kappa^{(3)}\circ\kappa^{(2)}_x$. Let $\mu_3'=(T_3)_\#\mu_{2}$ where $T_3:x\mapsto x^*$. Since by assumption  the mean of $\kappa^{(3)}_x$ lies in $[x-\lambda\ee/6,x+\lambda\ee/6]$, we must have $|x-x^*|\leq \lambda\ee/6$. Therefore, $W_\infty(\mu_2,\mu_3')<\lambda\ee/6$. 
 Consider the map $x^*\mapsto \E[\kappa^{(3)}\circ\kappa^{(2)}_{X_2}|T_3(X_2)=x^*]$ that aggregates $\kappa^{(3)}\circ\kappa^{(2)}_x$ among all sources $x$ such that $T_3(x)=x^*$, where $X_2\lawis\mu_2$. Since such a map
 forms a martingale transport from $\mu_3'$ to $\mu_4$, it follows that
 $\mu_3'\cx\mu_4$.

 It now suffices to prove $\mu_1\cx\mu_3'$. Consider a coupling $(X_1,X_2,X_3)$ such that $X_i\lawis\mu_i$ for $i=1,2$ and $X_3\lawis\mu_3'$, $X_2=X_1+\ee B\ep$, and $|X_2-X_3|<\lambda\ee/6$.
Let $a\in\R$; we will show that $\E[(X_1-a)_+]\leq \E[(X_3-a)_+]$. The case $a>x_0+\lambda\ee/6$ is obvious. If $a\in[x_0,x_0+\lambda\ee/6]$, we have using $|X_1-x_0|\leq \lambda\ee/6$ that
$$\E[(X_1-a)_+]\leq \frac{\lambda\ee}{6}\leq \frac{\lambda}{2}\left(\ee-\frac{3\lambda\ee}{6}\right)\leq \E[(X_3-a)_+]. $$ The other cases are symmetric using our assumption $\E[X_1]=x_0=\E[X_3]$.
\end{proof}

\begin{proof}[Proof of Theorem~\ref{th:dense} for general $\mu$] Let $\mu\cx\nu$ with $\nu$ atomless, $\pi_0\in\pim$ and $\ee>0$. Consider quantities $\delta,\lambda\in(0,1)$ small enough (to be determined below) depending on $\ee$. Define  $$A_\delta =[-\delta,\delta]+\{x\in \R:0\leq u_\nu(x)-u_\mu(x)<\delta\}$$ which appears in Lemma  \ref{lem:june-2},
as well as $A_\delta' =\{x\in \R:0\leq u_\nu(x)-u_\mu(x)<\delta\}$. We divide the rest of the proof into three steps.

\paragraph{Step I: inserting Rademacher noise.} Let $X\lawis \mu$ and $\ep\lawis \mathrm{Rade}$ be  independent. 
Denote by $\tilde \mu$ the distribution of $X_\delta := X+\delta \ep\bone_{\{X\not \in A_\delta\}}$.
We have  $ \mu \le_{\rm cx}  \tilde \mu $. Observe that for $x\not\in A_{\delta}'$, $u_\nu(x)\geq u_\mu(x)+\delta$, so that $u_\nu(x)\geq u_{\tilde\mu}(x)$ by the triangle inequality. For  $x\in A_{\delta}'$, we have 
\begin{align*}
    u_{\tilde\mu}(x)&=\E[|X-x|\bone_{\{X\in A_\delta\}}]+\E[|X+\delta\ep-x|\bone_{\{X\not\in A_\delta\}}]\\
    &=\E[|X-x|\bone_{\{X\in A_\delta\}}]+\E[|X-x|\bone_{\{X\not\in A_\delta\}}]=u_\mu(x)\leq u_\nu(x),
\end{align*}where in the second equality we used that $|X-x|\geq \delta$ on the set $\{X\not\in A_\delta\}$, by definition of $A_\delta,A_\delta'$. As a result, $\mu\cx\tilde\mu\cx\nu$.

Let $\pi_1$ be any martingale transport between 
$\tilde \mu$ and $\nu$,
and $\pi_2$ be the martingale transport given by $(X,X_\delta)$. 
Note that the kernel of $\pi_2$ has support $\{-\delta,\delta\}$ on $\R\setminus A_\delta$ and is the identity kernel on~$A_\delta$. 
Composing $\pi_2$ and $\pi_1$ we get a coupling from $\mu$ to $\nu$, denoted $  \pi_3$. Let $\tilde \pi_0= (1-\lambda)\pi_0+ \lambda \pi_3$.  It then suffices to approximate $\tilde\pi_0$ instead of $\pi_0$, i.e., to show that $\tilde \pi_0$ belongs to the weak closure of $\pib$. Once that is shown, it will follow by taking $\lambda\to 0$ that $\pi_0$ is also in the closure.

\paragraph{Step II: decomposition of the measures.} Partition $\R$ into  intervals $\{I_\ell\}_{\ell\in \N}$ such that $|I_\ell|\le \lambda \ee/6$, where $|I|$ denotes the length of an interval~$I$. Let us discard all $I_\ell$ with $\nu(I_\ell)=0$. 
We also partition $\R\setminus A_\delta$ into  intervals $\{J_k\}_{k\in \N}$ such that $|J_k|\le \lambda\ee/6$, and define $J_0=A_\delta$.
Note that this is possible since $A_\delta$ is the union of  some intervals. Again, let us discard all $J_k$ with $\mu(J_k)=0$.

Next, focus on one interval $I_\ell$. Let $\N_0$ denote the set of nonnegative integers. For $k\in \N_0$, consider the image of~$\mu|_{J_k}$ 
under~$\tilde \pi_0$ 
which we denote by $\tilde \nu_k$. Moreover, let $\tilde \nu_{k,\ell} =\tilde \nu_k|_{I_\ell}$ for $k\in\N_0$.
Note that $\{\tilde \nu_{k,\ell}\}_{k\in \N_0}$ forms a decomposition of $\nu|_{I_\ell}$.
Applying  Lemma~\ref{lem:june-1} to this decomposition, 
we obtain mutually singular $\{\hat \nu_{k,\ell}\}_{k\in \N_0}$ such that 
$\nu|_{I_\ell}=\sum_{k=1}^\infty  \hat \nu_{k,\ell}$, 
 $\bary(\tilde \nu_{k,\ell})=\bary(\hat \nu_{k,\ell})$ for $k\in \N$, $W_\infty(\tilde \nu_{k,\ell},\hat \nu_{k,\ell})\leq |I_\ell|$, and $\tilde \nu_{0,\ell}\le_{\rm cx} \hat \nu_{0,\ell}$; see Figure \ref{fig:5} below for an illustration.  
 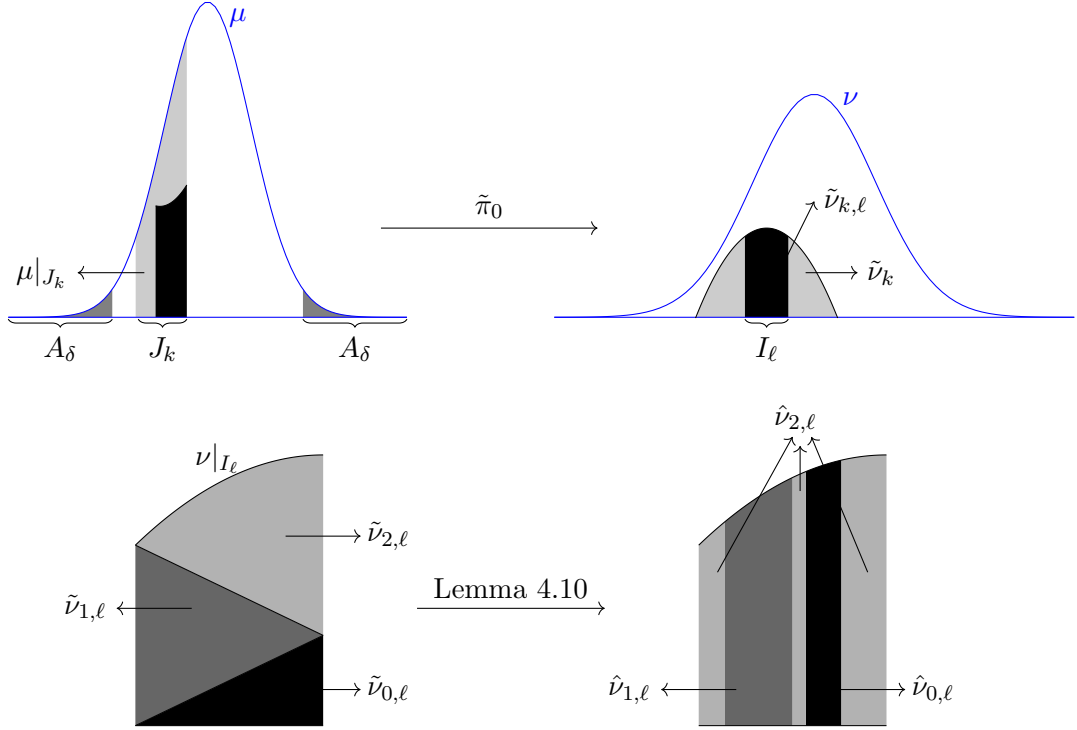
\begin{figure}[hbtp]\begin{center} 
 \begin{tikzpicture}
\centering
\begin{axis}[width=\textwidth,
      height=0.4\textwidth,
  axis lines=none,%
  domain=-10:10,
  xmin=-3, xmax=10,
  ymin=-0.2, ymax=1.5,
  samples=80,
]
  \addplot[name path=f,blue,domain={-2.3:2.3},mark=none]
    {1.414*exp(-2*x^2)}node[pos=.55, above]{$\ \ \mu$};
  \addplot[name path=g,blue,domain={-2.3:2.3},mark=none]
    {0};
     \addplot[name path=h,black,domain={-0.6:-0.24},mark=none]
    {1.1*x+0.8+x^2};
   
    \addplot[black!50]fill between[of=f and g, soft clip={domain=-2.3:-1.1}];
    \addplot[black!50]fill between[of=f and g, soft clip={domain=1.1:2.3}];
    \addplot[black!20]fill between[of=f and g, soft clip={domain=-0.83:-0.24}];
     \addplot[black!100]fill between[of=h and g, soft clip={domain=-0.6:-0.24}];
  
   \draw [decorate,decoration={brace,mirror,amplitude=3pt},xshift=0pt,yshift=1pt]
(7,18) -- (19,18) node [black,midway,yshift=-0.4cm]  { $A_\delta$};
\draw [decorate,decoration={brace,mirror,amplitude=3pt},xshift=0pt,yshift=1pt]
(41,18) -- (53,18) node [black,midway,yshift=-0.4cm]  { $A_\delta$};
  \addplot[name path=p,blue,domain={4:10}]
    {exp(-(x-7)^2)}node[pos=.55, above]{$\ \ \nu$};
  \addplot[name path=q,blue,domain={4:10}]
    {0};
     \addplot[name path=pp,black,domain={5.63:7.27}]
    {0.4-0.6*(x-6.45)^2};
    
    \addplot[black!20]fill between[of=pp and q, soft clip={domain=5.65:7.25}];
    \addplot[black!100]fill between[of=pp and q, soft clip={domain=6.2:6.7}];
    \draw[->](99,40)--(105,40)node[pos=.99, right]{$\tilde\nu_k$};
    \draw[->](94,25)--(100,72)node[pos=.99, right]{$\tilde\nu_{k,\ell}$};
    
\draw[->](22.7,40)--(15,40)node[pos=.1, left]{\hspace{-2.5cm}$\mu|_{J_k}$};
\draw[->](50,60)--(75,60)node[pos=.5, above]{$\tilde{\pi}_0$};

\draw [decorate,decoration={brace,mirror,amplitude=3pt},xshift=0pt,yshift=1pt]
(22,18) -- (27.6,18)  node[black,midway,yshift=-0.4cm] {$J_k$};
\draw [decorate,decoration={brace,mirror,amplitude=3pt},xshift=0pt,yshift=1pt]
(92,18) -- (97,18) node [black,midway,yshift=-0.4cm]  { $I_\ell$};
\end{axis}

\end{tikzpicture}

\begin{tikzpicture}
\centering
\begin{axis}[width=\textwidth,
      height=0.4\textwidth,
  axis lines=none,%
  domain=-3:3,
  xmin=-3, xmax=3,
  ymin=-0.2, ymax=4,
  samples=80,
]
\addplot[name path=a,black,domain={-2:-1},mark=none]
    {0};
  \addplot[name path=b,black,domain={-2:-1},mark=none]
    {x+2};
    \addplot[name path=c,black,domain={-2:-1},mark=none]
    {-x};
    \addplot[name path=d,black,domain={-2:-1},mark=none]
    {-(x+1)^2+3}node[pos=.55, above]{$\nu|_{I_\ell}$};

    \addplot[black!100]fill between[of=a and b, soft clip={domain=-2:-1}];
    \addplot[black!60]fill between[of=c and b, soft clip={domain=-2:-1}];
    \addplot[black!30]fill between[of=c and d, soft clip={domain=-2:-1}];
    
    \draw[->](180,230)--(220,230)node[pos=.95, right]{$\tilde\nu_{0,\ell}$};
    \draw[->](130,150)--(90,150)node[pos=.04, left]{\hspace{-2.7cm}$\tilde\nu_{1,\ell}$};
    \draw[->](180,60)--(220,60)node[pos=.95, right]{$\tilde\nu_{2,\ell}$};
\draw[->](250,150)--(350,150)node[pos=.5, above]{Lemma~\ref{lem:june-1}};
    
    \addplot[name path=e,black,domain={1:2},mark=none]
    {0};
    \addplot[name path=k,black,domain={1:2},mark=none]
    {-(x-2)^2+3};

    \addplot[black!100]fill between[of=e and k, soft clip={domain=1.57:1.76}];
    \addplot[black!60]fill between[of=e and k, soft clip={domain=1.14:1.5}];
    \addplot[black!30]fill between[of=e and k, soft clip={domain=1:1.14}];
    \addplot[black!30]fill between[of=e and k, soft clip={domain=1.76:2}];
    \addplot[black!30]fill between[of=e and k, soft clip={domain=1.5:1.57}];
    
    \draw[->](470,60)--(510,60)node[pos=.95, right]{$\hat\nu_{2,\ell}$};
     \draw[->](420,60)--(380,60)node[pos=.95, right]{\hspace{-1cm}$\hat\nu_{1,\ell}$};
    \draw[->](410,190)--(450,340);
    \draw[->](454,280)--(454,330);
    \draw[->](490,190)--(460,340)node[pos=.95, above]{\hspace{-0.5cm}$\hat\nu_{0,\ell}$};
    
\end{axis}\end{tikzpicture}
\end{center}
\caption{Illustrating the transport $\tilde{\pi}_0$ and Lemma~\ref{lem:june-1}}
\label{fig:5}
\end{figure}
Recall the definitions of $\pi_3$ and $\tilde\pi_0$.
 
 \begin{enumerate}[(a)]
     \item Applying Lemma~\ref{l1} with $\mu_1=\mu|_{J_k},\ \mu_2$ the image of $\mu_1$ under the transport $(1-\lambda)\text{id}+\lambda\pi_2$, $\mu_3=\tilde{\nu}_k=\sum_{\ell\in \N} \tilde \nu_{k,\ell}$, and $\mu_4=\sum_{\ell\in \N} \hat \nu_{k,\ell}$
     while noting that $$W_\infty(\mu_3,\mu_4)\leq \sup_{\ell\in\N}W_\infty(\tilde \nu_{k,\ell},\hat \nu_{k,\ell})\le \sup_{\ell\in\N}\tilde \nu_{k,\ell} (\R)|I_\ell| \leq \frac{\lambda\ee}{6},$$ we conclude that  $\mu |_{J_k} \le_{\rm cx} \sum_{\ell\in \N} \hat \nu_{k,\ell}$ for $k\in \N$.
     \item Similarly, it follows that $\mu|_{J_0}=\mu|_{A_\delta}\le_{\rm cx} \tilde \nu_0= \sum_{\ell\in \N}\tilde \nu_{0,\ell}\cx  \sum_{\ell\in \N}\hat \nu_{0,\ell} $.
 \end{enumerate}

 \paragraph{Step III: approximating $\tilde\pi_0$ by MMT.}
 We can now construct an approximation $\hat{\pi}$ of $\tilde{\pi}_0$ as follows. Note that since $\nu$ is atomless, so is $\hat\nu_{k,\ell}$ for all $k\in\N_0,$ $\ell\in\N$.
\begin{enumerate}[(a)]
    \item For each $k\in \N$, applying Theorem~\ref{thm1} to $\mu |_{J_k}$  and $\sum_{\ell\in \N} \hat \nu_{k,\ell}$ yields a coupling $  \hat \pi^k$ which is an MMT between  $\mu |_{J_k}$  and $\sum_{\ell\in \N} \hat \nu_{k,\ell}$. Denote by $\tilde \pi^k$
the original coupling between $\mu|_{J_k}$  and  $\tilde \nu_k$
induced by $\tilde \pi_0$. It follows that
\begin{align*}W_1 \left(\sum_{k\in \N}   \tilde \pi^k, \sum_{k\in \N} \hat \pi^k \right)&\leq W_\infty \left(\sum_{k\in \N}   \tilde \pi^k, \sum_{k\in \N} \hat \pi^k \right) \\
&\le\sup_{k\in \N}W_\infty(\tilde \pi^k, \hat \pi^k )\leq \sup_{k\in \N}\mu(J_k)\left(|J_k|+\max_{\ell\in \N}|I_\ell|\right)\leq  \lambda \ee.\end{align*}

\item We apply Theorem~\ref{thm1} to $\mu |_{A_\delta}$  and $ \sum_{\ell\in \N}\hat \nu_{0,\ell} $, and get another MMT, denoted $\hat \pi^0$. Denote by $\tilde\pi^0$ the original coupling between $\mu|_{A_\delta}$  and  $\tilde \nu_0$
induced by $\tilde \pi_0$. By Lemma~\ref{lem:june-2},  $\mu(A_\delta)\to  0$ as $\delta\to 0$, so that $W_1(\tilde \pi^0, \hat \pi^0)\to 0$.
\end{enumerate}
Since  $\{\hat\nu_{k,\ell}\}_{k\in\N_0,\ell\in\N}$ are mutually singular as noted above, it follows that $\hat \pi:=\sum_{k=0}^\infty \hat \pi^k$ is an MMT. The first marginal of $\hat \pi$ is $\mu|_{A_\delta} +\sum_{k=1}^\infty \mu|_{J_k}= \mu$
and the second marginal of  $\hat \pi$ is $ \sum_{k=0}^\infty \sum_{\ell\in \N} \hat \nu_{k,\ell}= \sum_{\ell\in \N}  \nu|_{I_\ell}=\nu$.
Therefore, $\hat \pi\in\pib$. Note that 
  $$W_1(\tilde \pi_0, \hat \pi)\le W_1(\tilde \pi^0, \hat \pi^0)  + W_1 \left (\sum_{k\in \N}   \tilde \pi^k, \sum_{k\in \N} \hat \pi^k \right).
  $$
  As shown above, both terms tend to $0$. Since $W_1$ convergence implies weak convergence, we conclude that $\pib$ is weakly dense in $\pim$.
\end{proof}

\subsection{General results on the uniqueness of MT and MMT}\label{U}

In this subsection, we characterize the uniqueness of martingale transports and Monge martingale transports using shadow measures, for general marginals $\mu,\nu\in\cP(\R)$ with $\mu \le_{\rm cx} \nu$ (possibly with atoms).
To the best of our knowledge, the uniqueness of MT has not been completely characterized, except for a few simple examples mentioned in \cite{DM18} and \cite{OS17}. 
The first result states that $\pim$ is a singleton if and only if the shadows of any decomposition of $\mu$ do not affect each other.

\begin{proposition}\label{prop:shadowpasting}
     The MT between $\mu$ and $\nu$ is unique if and only if   $\nu=\sum_{j=1}^n S^{\nu}(\mu_j)$ for any $n\in \N$ and mutually singular $\mu_1,\dots,\mu_n\le \mu$ satisfying $\sum_{j=1}^n \mu_j = \mu.$
\end{proposition} 

\begin{proof}
    We first show the ``if"  statement.
    Suppose that $\nu=\sum_{j=1}^n S^{\nu}(\mu_j)  $.
    We claim that the only possible MT is to transport $\mu_i$ to $S^{\nu}(\mu_i)$ for each $i$. Suppose otherwise, and let $\nu_i$ be the image of $\mu_i$ under a different MT. Then, by the minimality property of the shadow, there exist~$i$ and a convex function~$\phi$ such that 
    $
    \int \phi~ \d \nu_i > \int \phi~ \d S^{\nu}(\mu_i)
    $.
    As $
    \sum_{j=1}^n \int \phi~ \d \nu_j  = \int \phi~ \d \nu =  \sum_{j=1}^n \int \phi~ \d S^{\nu}(\mu_j),
    $
    it follows that there exists $j$ with  $
    \int \phi~ \d \nu_j < \int \phi~ \d S^{\nu}(\mu_j),
    $
violating the definition of the shadow. %

To show the  ``only if" statement, suppose that $\nu \ne\sum_{j=1}^n S^{\nu}(\mu_j) $ for some mutually singular $\mu_1,\dots,\mu_n$ adding up to $\mu$. Note that necessarily $n\geq2$ and fix $j\in \{1,\dots,n\}$. We define $\pi_j\in\pim$ by first transporting $\mu_j$ to $S^\nu(\mu_j)$, then removing $S^\nu(\mu_j)$ from $\nu$, and continuing in the same way for $\mu_{j+1},\dots,\mu_{n},\mu_{1},\dots,\mu_{j-1}$. 
If $\pi_1,\dots,\pi_n$ all coincide, then as the image of $\mu_j$ under $\pi_j$ is $S^\nu(\mu_j)$, we have
$\nu =\sum_{j=1}^n S^{\nu}(\mu_j)$, a contradiction.
\end{proof}

The second result further characterizes when the singleton~$\pim$ consists of an MMT.
    
\begin{proposition}\label{prop:shadowpasting2}
 The MT between $\mu$ and $\nu$ is unique and is an MMT if and only if  
 for any $n\in \N$ and mutually singular $\mu_1,\dots,\mu_n\le \mu$, the shadows
$ S^{\nu}(\mu_1), \dots, S^{\nu}(\mu_n)$ are mutually singular. 
\end{proposition}
     
\begin{proof}
     We first show the ``if"  statement. Suppose that $\mu_1,\dots,\mu_n\le \mu$ are mutually singular and satisfy $\sum_{j=1}^n \mu_j = \mu.$ If  $ S^{\nu}(\mu_1), \dots, S^{\nu}(\mu_n)$ are mutually singular, then  $\nu=\sum_{j=1}^n S^{\nu}(\mu_j)$ and  Proposition~\ref{prop:shadowpasting} shows that the MT is unique.
Next, we show that this MT is an MMT. 
As seen in the proof of Proposition~\ref{prop:shadowpasting}, 
the MT transports any $\mu'\le \mu$ to $S^\nu(\mu')$.
For $N\in \N$, we divide $\R$ into countably many disjoint subsets $A^N_i$, $i\in \N$, each of length $1/N$.
The mutual singularity assumption ensures that the set $B^N$ of points $y$ which transport (in the $\nu\to\mu$ direction) to at least two different subsets in $ \{A^N_i:i\in \N\}$ is $\nu$-negligible.
Thus, $\nu(\bigcup_{N\in \N}B^N) =0$, showing that the set of points~$y$ that map  to a single~$x$ has $\nu$-measure~$1$. In other words, the MT is an MMT. 
 
To see the ``only if" statement, let $\mu_1,\dots,\mu_n\le \mu$ be mutually singular. We may assume that $\sum_{j=1}^n \mu_j = \mu.$
Suppose that the MT is unique, then $\nu=\sum_{j=1}^n S^{\nu}(\mu_j)  $ by Proposition~\ref{prop:shadowpasting}. 
If   $S^{\nu}(\mu_1)$ and   $ S^{\nu}(\mu_2)$ are not mutually singular, then points in their common part must be transported to two disjoint sets supporting~$\mu_1$ and~$\mu_2$, so that this MT is not an MMT.
    \end{proof}

 As seen in Example~\ref{ex:uniq}, uniqueness of MMT does not  imply uniqueness of MT. Therefore, uniqueness of MMT is not sufficient for the conditions in Proposition~\ref{prop:shadowpasting} or Proposition~\ref{prop:shadowpasting2}.

\subsection{Proof of Theorem~\ref{th:uniq}}

Continuing the study of uniqueness, we now aim to characterize the uniqueness of MMT and MT more explicitly for $\mu \le_{\rm cx} \nu$ with $\nu$ atomless.

\begin{lemma}\label{5}
Suppose that $\nu$ is atomless and there is a unique MMT. For any $\gamma_1=a_{1}\delta_{x_{1}},\gamma_2=a_{2}\delta_{x_{2}}$ with $x_{1}\neq x_{2}$ and $\gamma_1+\gamma_2\leq \mu$, we have 
$S^{\nu-S^\nu(\gamma_1)}(\gamma_2)=S^\nu(\gamma_2)$.
In particular, $S^\nu(\gamma_1)$ and $S^\nu(\gamma_2)$ are restrictions of $\nu$ to disjoint intervals.
\end{lemma}

\begin{proof}
Recall that shadows are associative (Lemma~\ref{ass}). %
As $\gamma_1\cx S^\nu(\gamma_1)$ and $\gamma_2\cx S^{\nu-S^\nu(\gamma_1)}(\gamma_2)$, by Theorem~\ref{thm1} we obtain two MMTs, say $\pi_1$ from   $\gamma_1$ to $S^\nu(\gamma_1)$ and $\pi_2$ from  $\gamma_2$ to $S^{\nu-S^\nu(\gamma_1)}(\gamma_2)$. Moreover, $\mu-\gamma_1-\gamma_2\cx \nu-S^\nu(\gamma_1)-S^{\nu-S^\nu(\gamma_1)}(\gamma_2)=\nu-S^\nu(\gamma_1+\gamma_2)$, yielding another MMT~$\pi_3$ from  $\mu-\gamma_1-\gamma_2$ to  $\nu-S^\nu(\gamma_1+\gamma_2)$. By Lemma~\ref{ms}, the measures $S^\nu(\gamma_1),\ S^{\nu-S^\nu(\gamma_1)}(\gamma_2)$ and $\nu-S^\nu(\gamma_1+\gamma_2)$ are mutually singular. Thus, we may aggregate $\pi_i$, $i=1,2,3$ to get an MMT $\pi$ from $\mu$ to $\nu$.

Repeat the above construction switching the roles of $\gamma_1,\gamma_2$. The resulting MMT $\pi'$ transports $\gamma_2$ to $S^\nu(\gamma_2)$. As $\pi$ transports $\gamma_2$ to $S^{\nu-S^\nu(\gamma_1)}(\gamma_2)$ and $\pi=\pi'$ by the assumed uniqueness, we conclude $S^{\nu-S^\nu(\gamma_1)}(\gamma_2)=S^\nu(\gamma_2)$. The last statement then follows from Lemma~\ref{ms}.
\end{proof}

\begin{proof}[Proof of Theorem~\ref{th:uniq}] Clearly (i) implies (ii). To see that (ii) implies (iii), suppose that the MMT from $\mu$ to $\nu$ is unique. Consider the atomic part $\mu_a:=\sum_{j\in \N} a_j\delta_{x_j}$~of $\mu$ where the~$x_j$ are distinct.  Applying Lemma~\ref{5} with $\gamma_1=a_j\delta_{x_j}$ and $\gamma_2=a_{j'}\delta_{x_{j'}}$ yields that the shadows $S^\nu(a_j\delta_{x_j})$ are restrictions of $\nu$ to disjoint intervals. Removing $\mu_a$ and its shadow, we may thus assume that $\mu$ is atomless and prove $\mu=\nu$.  Suppose that $\mu\neq\nu$. There exists an interval $[a,b]$ such  that $\mu([a,b])> \nu([a,b])$. More precisely, we can find $a<b$ and $\ee_1,\ee_2>0$ such that 
\[
  0< \nu([a-\ee_1,a]),\nu([b,b+\ee_2]) < \frac{\mu([a,b])-\nu([a,b])}{2} \quad\mbox{and}\quad \mu([a-\ee_1,a]),\mu([b,b+\ee_2])>0.
\]
The minimality property of the shadow implies that either (a) $\nu|_{[a-\ee_1,a]}\leq S^\nu(\mu|_{[a,b]})$ or (b) $\nu|_{[b,b+\ee_2]}\leq S^\nu(\mu|_{[a,b]})$. Suppose that (a) holds. Similarly as in the proof of Lemma~\ref{5}, taking shadow first on $\mu|_{[a,b]}$ and then on $\mu|_{[a-\ee_1,a]}$, or vice versa, yields different MMTs, a contradiction. Case~(b) is analogous, thus (ii) implies (iii).

Suppose that (iii) holds and consider mutually singular $\mu_1,\dots,\mu_n\leq \mu$ satisfying $\sum_{j=1}^n\mu_j=\mu$. Decompose them into an atomic part $\mu^a_j$ and a continuous part $\mu^c_j$.  Then by (iii) and Lemma~\ref{atom}, $S^\nu(\mu_j^a)$ and $S^\nu(\mu_j^c)=\mu_j^c$ are mutually singular, and these are mutually singular for distinct $j$'s because $(\mu_j^a)_{1\leq j\leq n}$ are mutually singular. This implies $\nu=\sum_{j=1}^nS^\nu(\mu_j)$. Thus Proposition~\ref{prop:shadowpasting} shows that (i) holds, completing the proof.
\end{proof}

\section{Concluding remarks}
\label{sec:concluding}

In this section, we briefly discuss some open problems.

\paragraph{MMT in higher dimensions.}
The present paper focuses on martingale transport on~$\R$. Starting with 
\cite{GKL19}, \cite{OS17}, and \cite{DeMarchTouzi.17}, martingale transport in~$\R^{d}$ has been actively studied in the recent literature, but is well known to be intricate. See, e.g., \cite{WZ22} for further references. We continue to use $\mathcal M(\mu,\nu)$  (resp.~$\mathcal M_{M}(\mu,\nu)$) for the set of all martingale (resp.~Monge martingale) transports between $\mu$ and $\nu$.

A crucial ingredient in analyzing martingale transport in higher dimensions is the irreducible decomposition, which disintegrates the martingale transport problem into irreducible components. Following \cite{DeMarchTouzi.17}, let $\hat{\mathcal{K}}$ be the set of all convex closed subsets of $\R^d$. For probability measures $\mu,\nu$ on $\R^d$, the irreducible components map $I:\R^d\to\hat{\mathcal{K}}$ is the ($\mu$-a.e.~unique) map such that for some $\hat{\P}\in\pim$, $\ri\conv\supp \P_X\subseteq I(X)=\ri\conv\supp \hat{\P}_X$ holds $\mu$-a.e.~(where $X\lawis\mu$ and $\{\P_x\}_{x\in\R^d}$ is the disintegration of $\P$), for all $\P\in\pim$. Moreover, $\{I(x):x\in\R^d\}$ forms a partition of $\R^d$. We may further disintegrate $\nu$ into $\{\nu_x:x\in\R^d\}$ along such a partition.

\begin{conjecture}\label{conj}
    Let $\mu,\nu$ be probability measures on $\R^d$ satisfying $\mu\cx\nu$. Suppose that $\nu_x$ is atomless for $\mu$-a.e.~$x\in\R^d$. Then $\pib$ is weakly dense in $\pim$. If $\mu$ is discrete, it is also dense for the $\infty$-Wasserstein topology.
\end{conjecture}

In particular, an analogue of the existence result in Theorem~\ref{thm1} may pave the path to a denseness result along the lines of Theorem~\ref{th:dense} with similar proof ideas.
The main difficulty in proving Conjecture \ref{conj} lies in constructing a suitable analogue of the left-curtain coupling in higher dimensions.
Note also that in dimension $d=1$, the irreducible decomposition (cf.~Remark \ref{remark:referee}) is countable, so assuming non-atomicity before the irreducible decomposition is sufficient. 
The following remark shows that the absence of atoms (before the irreducible decomposition) is not sufficient for existence in dimensions $d>1$. 

\begin{remark}
Na\"{i}ve analogues of Theorem~\ref{thm1} and Theorem~\ref{th:refinedStrassen} in $\R^d$, assuming only that the  marginals are atomless, are false. Let $\mu$ be uniform on $[0,1]\times\{0,\pm 1\}$ and $\nu$ uniform on $[0,1]\times\{\pm 2\}$. Let 
$(X,Y)=((X_1,X_2),(Y_1,Y_2))$ be a martingale transport; then $(X_{1},Y_{1})$ is a martingale with both marginals $\Unif[0,1]$, so that $X_{1}=Y_{1}$. Moreover, $(X_{2},Y_{2})$ is the unique (in law) martingale from $\Unif\{0,\pm 1\}$ to $\Unif\{\pm 2\}$. We see that $\pim$ is a singleton, and this martingale transport is clearly not (backward) Monge. In the language of \cite{DeMarchTouzi.17}, the irreducible decomposition corresponds to disintegration along the first coordinate; cf.~Example 2.2 of \cite{OS17}.  As seen in Remark~\ref{StrassenAndExistence}, this non-existence of an MMT also precludes the assertion of Theorem~\ref{th:refinedStrassen}.
\end{remark}

\paragraph{Denseness results under different constraints.}
Going back to transports on~$\R$, let us turn to a different generalization, namely the constraint. We have seen that martingale transports typically do not admit Monge maps in the forward direction and that the left-curtain transport is supported on the union of two graphs. These facts are due to the martingale constraint. Similar phenomena arise for other constraints, in particular the supermartingale constraint $\E[Y|X]\leq X$ of \cite{NS18}, \cite{BayraktarDengNorgilas.21,BayraktarDengNorgilas.22} and the directional constraint $X\leq Y$ of \cite{NW22}.
A supermartingale coupling between $\mu$ and $\nu$ exists if and only if $\mu\leq_{\rm cd}\nu$ (meaning that $\int\phi\,\d\mu\leq \int\phi\,\d\nu$ for all convex decreasing $\phi$), and a coupling $(X,Y)$ of $\mu$ and $\nu$ satisfying the directional constraint $X\leq Y$ exists if and only if $\mu\leq_{\rm st}\nu$ (meaning that their cdfs satisfy $F_\mu\geq F_\nu$).
We speculate that, in analogy with Theorem~\ref{th:dense}, the set of constrained (backward) Monge transports is dense also in those settings, and possibly for other constraints.

\begin{conjecture}
    Let $\mu\leq_{\rm cd}\nu$ with $\nu$ atomless. Then the set of (backward) Monge supermartingale couplings is weakly dense in the set of supermartingale couplings between $\mu$ and $\nu$. 
\end{conjecture}

\begin{conjecture}
    Let $\mu\leq_{\rm st}\nu$ with $\nu$ atomless. Then the set of (backward) Monge couplings $(X,Y)$ satisfying $X\leq Y$ is weakly dense in the set of all couplings $(X,Y)$ satisfying $X\leq Y$ between $\mu$ and $\nu$. 
\end{conjecture}

\section*{Acknowledgements}
We thank two anonymous referees for their careful reading of a previous version of this paper and for pointing out Remark \ref{remark:referee}.

\end{document}